\documentclass[10pt,leqno]{article} 

   \usepackage[centertags]{amsmath}
   \usepackage{amsfonts}
   \usepackage{amsmath}
   \usepackage{amssymb}
   \usepackage{amsthm}
   \usepackage{newlfont}
   \usepackage[latin1]{inputenc}

 \usepackage{multirow, bigdelim}

\allowdisplaybreaks[3]

\theoremstyle{plain}
\newtheorem{theorem}{Theorem}[section]
\newtheorem{lemma}[theorem]{Lemma}

\newtheorem{corollary}[theorem]{Corollary}

\theoremstyle{definition}
\newtheorem{definition}[theorem]{Definition}

\theoremstyle{remark}

\newtheorem*{remark*}{Remark}

\numberwithin{equation}{section}

\newcommand{\dosfilas}[2]{
  \ldelim[{2}{2mm}& #1 &\rdelim]{2}{2mm} \\
  & #2 & &  & &
}

\catcode`,\active

\catcode`\,12

\newcommand\D{{\mathcal D}}

\newcommand\F{{\mathcal F}}
\newcommand\G{{\mathcal G}}

\newcommand\CC{{\mathbb C}}

\newcommand\RR{{\mathbb R}}

\newcommand\NN{{\mathbb N}}
\newcommand\PP{{\mathbb P}}
\newcommand\Aa{{\mathbb A}}
\newcommand\Bb{{\mathbb B}}

\newcommand\X{{\Theta}}
\newcommand\x{{\theta}}

\newcommand\p{\mbox{$\mathfrak{p}$}}
\newcommand\pp{\mbox{$\mathfrak{p}_{F}$}}

\newcommand\rank{\operatorname{rank}}

\newcommand\sign{\operatorname{sign}}

\newcommand\Sh{\mbox{\Large $\mathfrak {s}$}}

   \title{Exceptional Meixner and Laguerre orthogonal polynomials
  \footnote{Partially supported by MTM2012-36732-C03-03 (Ministerio de Economía y Competitividad),
FQM-262, FQM-4643, FQM-7276 (Junta de Andalucía) and Feder Funds (European
Union).}}
   \author{Antonio J. Dur\'{a}n\\
     \footnotesize
        \  Departamento de An\'{a}lisis Matem\'{a}tico.
       Universidad de Sevilla \\
       \footnotesize Apdo (P. O. BOX) 1160. 41080 Sevilla. Spain.
   duran@us.es \\
          \ \ }
   \date{}
   
\begin{document}
   \maketitle

\bigskip

\begin{abstract}
Using Casorati determinants of Meixner polynomials $(m_n^{a,c})_n$, we construct for each pair $\F=(F_1,F_2)$ of finite sets of positive integers a sequence of polynomials $m_n^{a,c;\F}$, $n\in \sigma _\F$, which are eigenfunctions of a second order difference operator, where $\sigma _\F$ is certain infinite set of nonnegative integers, $\sigma _\F \varsubsetneq \NN$.  When $c$ and $\F$ satisfy a suitable admissibility condition, we prove that the polynomials $m_n^{a,c;\F}$, $n\in \sigma _\F$, are actually exceptional Meixner polynomials; that is, in addition, they are orthogonal and complete with respect to a positive measure. By passing to the limit, we transform the Casorati determinant of Meixner polynomials into a Wronskian type determinant of Laguerre polynomials $(L_n^\alpha)_n$. Under the admissibility conditions for $\F$ and $\alpha$, these Wronskian type determinants turn out to be exceptional Laguerre polynomials.
\end{abstract}

\section{Introduction}
In \cite{duch}, we have introduced a systematic way of constructing exceptional discrete orthogonal polynomials using the concept of dual families of polynomials. Using Charlier polynomials, we applied  this procedure to construct exceptional Charlier polynomials and, passing to the limit, exceptional Hermite polynomials. The purpose of this paper is to extend this construction using  Meixner and Laguerre polynomials.

Exceptional orthogonal polynomials $p_n$, $n\in X\varsubsetneq \NN$, are complete orthogonal polynomial systems with respect to a positive measure which in addition
are eigenfunctions of a second order differential operator. They extend the  classical families of Hermite, Laguerre and Jacobi. The last few years have seen a great deal of activity in the area  of exceptional orthogonal polynomials (see, for instance,
\cite{DEK}, \cite{GUKM1}, \cite{GUKM2} (where the adjective \textrm{exceptional} for this topic was introduced), \cite{GUKM3}, \cite{GUKM4}, \cite{GUGM}, \cite{G}, \cite{GQ}, \cite{MR}, \cite{OS0}, \cite{OS}, \cite{OS3}, \cite{Qu}, \cite{STZ}, \cite{Ta} and the references therein).

In the same way, exceptional discrete orthogonal polynomials are complete orthogonal polynomial systems with respect to a positive measure which in addition are eigenfunctions of a second order difference operator, extending the discrete classical families of Charlier, Meixner, Krawtchouk and Hahn, or Wilson, Racah, etc., if orthogonal discrete polynomials on nonuniform lattices are considered (\cite{duch}, \cite{OS}, \cite{YZ}). One can also add to the list the exceptional $q$-orthogonal polynomials related to second order $q$-difference operators (\cite{OS,OS2,OS4,OS5,OS6}).

The most apparent difference between classical or classical discrete orthogonal polynomials and their exceptional counterparts
is that the exceptional families have gaps in their degrees, in the
sense that not all degrees are present in the sequence of polynomials (as it happens with the classical families) although they form a complete orthonormal set of the underlying $L^2$ space defined by the orthogonalizing positive measure. This
means in particular that they are not covered by the hypotheses of Bochner's and Lancaster's classification theorems (see \cite{B} or \cite{La}) for classical and classical discrete orthogonal polynomials, respectively.
Exceptional orthogonal polynomials have been applied to shape-invariant potentials \cite{Qu},
supersymmetric transformations \cite{GUKM3}, to discrete quantum mechanics \cite{OS}, mass-dependent potentials \cite{MR}, and to quasi-exact solvability \cite{Ta}.

As mentioned above, we use the concept of dual families of polynomials to construct exceptional discrete orthogonal polynomials (see \cite{Leo}). One can then also construct examples of exceptional orthogonal polynomials by taking limits in some of the parameters in the same way as one goes from classical discrete polynomials to classical polynomials in the Askey tableau.

\begin{definition}\label{dfp}
Given two sets of nonnegative integers $U,V\subset \NN$, we say that the two sequences of polynomials
$(p_u)_{u\in U}$, $(q_v)_{v\in V}$ are dual if there exist a couple of sequences of numbers $(\xi_u)_{u\in U}, (\zeta_v)_{v\in V} $ such that
\begin{equation}\label{defdp}
\xi_up_u(v)=\zeta_vq_v(u), \quad u\in U, v\in V.
\end{equation}
\end{definition}

Duality has shown to be a fruitful concept regarding discrete orthogonal polynomials, and his utility will be again manifest in the exceptional discrete polynomials world. Indeed, as we pointed out in \cite{duch}, it turns out that duality interchanges exceptional discrete orthogonal polynomials with the so-called Krall discrete orthogonal polynomials. A Krall discrete orthogonal family is a sequence of polynomials $(p_n)_{n\in \NN}$, $p_n$ of degree $n$, orthogonal with respect to a positive measure which, in addition, are also eigenfunctions of a higher order difference operator. A huge amount of families of Krall discrete orthogonal polynomials have been recently introduced by the author by mean of certain Christoffel transform of the classical discrete measures of Charlier, Meixner, Krawtchouk and Hahn (see \cite{du0}, \cite{du1}, \cite{DdI}). A Christoffel transform is a transformation which consists in multiplying a measure $\mu$ by a polynomial $r$. It has a long tradition in the context of orthogonal polynomials: it goes back a century and a half ago when E.B. Christoffel (see \cite{Chr} and also \cite{Sz}) studied it for the particular case $r(x)=x$.

The content of this paper is as follows. In Section 2, we include some preliminary results about Christoffel transforms and finite sets of positive integers.

In Section 3, using Casorati determinants of Meixner polynomials we associated to a pair $\F=(F_1,F_2)$ of finite sets of positive integers a sequence of polynomials which are eigenfunctions of a second order difference operator.

Denote by $\F=(F_1,F_2)$ a pair of finite sets of
positive integers, and write $k_i$ for the number of elements of $F_i$,
$i=1,2$ and $k=k_1+k_2$ for the number of elements of $\F$. One of
the components of $\F$, but not both, can be the empty set.
We define the nonnegative integer $u_\F$  by $u_\F=\sum_{f\in F_1}f+\sum_{f\in
F_2}f-\binom{k_1+1}{2}-\binom{k_2}{2}$ and the infinite set of nonnegative integers $\sigma _\F$ by
$$
\sigma _\F=\{u_\F,u_\F+1,u_\F+2,\cdots \}\setminus \{u_\F+f,f\in F_1\}.
$$
Given $a,c\in \RR $ with $a\not =0,1$ and $c\not =0,-1,-2,\ldots$, we then associate to the pair $\F$ the sequence of polynomials $m_n^{a,c;\F}$, $n\in \sigma _\F$, defined by
\begin{equation}\label{defmexi}
m_n^{a,c;\F}(x)=  \left|
  \begin{array}{@{}c@{}cccc@{}c@{}}
    & m_{n-u_\F}^{a,c}(x)&m_{n-u_\F}^{a,c}(x+1)&\cdots &m_{n-u_\F}^{a,c}(x+k) & \\
    \dosfilas{ m_{f}^{a,c}(x) & m_{f}^{a,c}(x+1) &\cdots  & m_{f}^{a,c}(x+k) }{f\in F_1} \\
    \dosfilas{ m_{f}^{1/a,c}(x) & m_{f}^{1/a,c}(x+1)/a & \cdots & m_{f}^{1/a,c}(x+k)/a^k }{f\in F_2}
  \end{array}
  \right|
\end{equation}
where $(m_n^{a,c})_n$ are the Meixner polynomials (see (\ref{Mxpol})). Along this paper, we use the following notation:
given a finite set of positive integers $F=\{f_1,\ldots , f_m\}$, the expression
\begin{equation}\label{defdosf}
  \begin{array}{@{}c@{}cccc@{}c@{}}
    \dosfilas{ z_{f,1} & z_{f,2} &\cdots  & z_{f,m} }{f\in F}
  \end{array}
\end{equation}
inside of a matrix or a determinant will mean the submatrix defined by
$$
\left(
\begin{array}{cccc}
z_{f_1,1} & z_{f_1,2} &\cdots  & z_{f_1,m}\\
\vdots &\vdots &\ddots &\vdots \\
z_{f_m,1} & z_{f_m,2} &\cdots  & z_{f_m,m}
\end{array}
\right) .
$$
The determinant (\ref{defmexi}) should be understood in this form.

When $-1<a<1$ Meixner polynomials $(m_n^{a,c})_n$ are orthogonal with respect to the discrete measure
$$
\rho_{a,c} =\sum_{x=0}^\infty \frac{a^x\Gamma(x+c)}{x!}\delta _x.
$$
Consider now the measure
\begin{equation}\label{ctmew}
\rho _{a,c}^{\F}=\prod_{f\in F_1}(x-f)\prod_{f\in F_2}(x+c+f)\rho _{a,c}.
\end{equation}
It turns out that the sequence of polynomials $m_n^{a,c;\F}$, $n\in \sigma _\F$, and the sequence of orthogonal polynomials with respect to the measure $\rho _{a,c}^{\F}$ are dual sequences (see Lemma \ref{lem3.2}). As a consequence we get that the polynomials $m_n^{a,c;\F}$, $n\in \sigma _\F$, are always eigenfunctions of a second order difference operator $D_\F$ (whose coefficients are rational functions); see Theorem \ref{th3.3}.

The most interesting case appears when the measure $\rho _{a,c}^{\F}$ is positive. This gives rise to the concept of admissibility for the real number $c$ and the pair $\F$, which we study in Section 2.4.

\begin{definition} Let $\F=(F_1,F_2)$ be a pair of finite sets of positive integers. For a real number $c\not=0,-1,-2,\cdots$, write
$\hat c=\max \{-[c],0\}$, where $[c]$ denotes the integral part of $c$. We say that $c$ and $\F$ are admissible if for all $x\in \NN $
\begin{equation}\label{defadmi}
\frac{\prod_{f\in F_1}(x-f)\prod_{f\in F_2}(x+c+f)}{(x+c)_{\hat c}}\ge 0.
\end{equation}
\end{definition}
Let us remind that for Charlier and Hermite polynomials, the admissibility of a finite set $F$ of positive integers is defined by $\prod_{f\in F}(x-f)\ge 0$, for $x\in \NN$. The concept of admissibility defined in (\ref{defadmi}) is more involve than the corresponding one for exceptional Charlier and Hermite polynomials because of two reasons. On the one hand, we have now a pair $\F$ of finite sets instead of a single finite set $F$. On the other hand, the admissibility also depends on the parameter $c$ of the Meixner polynomials (or on the parameter $\alpha$ of the Laguerre polynomials) while Charlier and Hermite admissibility only depends on the finite set $F$. The concept of admissibility for exceptional Charlier and Hermite polynomials has appeared several times in the literature (see, for instance, \cite{Kr}, \cite{Ad} or \cite{YZ}); however, we have not found in the literature a definition as (\ref{defadmi}) for Meixner and Laguerre admissibility.

In Section 4, we  prove (Theorems \ref{th4.4} and \ref{th4.5}) that if $c$ and $\F$ are admissible, then the polynomials $m_n^{a,c;\F}$, $n\in \sigma _\F$, are orthogonal and complete with respect to the positive measure
$$
\omega_{a,c}^{\F} =\sum_{x=0}^\infty \frac{a^x\Gamma(x+c+k)}{x!\Omega ^{a,c}_\F(x)\Omega ^{a,c}_\F(x+1)}\delta _x,
$$
where $\Omega _\F^{a,c}$ is the polynomial defined by
\begin{equation}\label{defmexii}
\Omega _\F^{a,c}(x)=  \left|
  \begin{array}{@{}c@{}cccc@{}c@{}}
    \dosfilas{ m_{f}^{a,c}(x) & m_{f}^{a,c}(x+1) &\cdots  & m_{f}^{a,c}(x+k-1) }{f\in F_1} \\
    \dosfilas{ m_{f}^{1/a,c}(x) & m_{f}^{1/a,c}(x+1)/a & \cdots & m_{f}^{1/a,c}(x+k-1)/a^{k-1} }{f\in F_2}
  \end{array}
  \right|.
\end{equation}
In particular we characterize the admissibility of $c$ and $\F$ in terms of the positivity of $\Gamma(x+c+k)\Omega ^{a,c}_\F(x)\Omega ^{a,c}_\F(x+1)$  for $x\in \NN$ (Lemma \ref{l3.1}).

In Section 5 and 6, we construct exceptional Laguerre polynomials by taking limit (in a suitable way) in the exceptional Meixner polynomials when $a\to 1 $. We then get (see Theorem \ref{th5.1}) that for each pair $\F=(F_1,F_2)$ of finite sets of positive integers, the polynomials
\begin{equation}\label{deflaxi}
L_n^{\alpha ;\F}(x)= \left|
  \begin{array}{@{}c@{}cccc@{}c@{}}
    & L_{n-u_\F}^{\alpha}(x)&(L_{n-u_\F}^{\alpha})'(x)&\cdots &(L_{n-u_\F}^{\alpha})^{(k)}(x) & \\
    \dosfilas{ L_{f}^{\alpha}(x) & (L_{f}^{\alpha})'(x) &\cdots  & (L_{f}^{\alpha})^{(k)}(x) }{f\in F_1} \\
    \dosfilas{ L_{f}^{\alpha}(-x) & L_{f}^{\alpha+1}(-x) & \cdots & L_{f}^{\alpha +k}(-x) }{f\in F_2}
  \end{array}
  \right|,
\end{equation}
$n\in \sigma _\F$, are eigenfunctions of a second order differential operator.

When $\alpha +1$ and $\F$ are admissible, we prove that $\alpha +k>-1$ and that the determinant $\Omega _F^\alpha $  defined by
\begin{equation}\label{deflaii}
\Omega _{\F}^{\alpha}(x)=\left|
  \begin{array}{@{}c@{}cccc@{}c@{}}
    \dosfilas{ L_{f}^{\alpha}(x) & (L_{f}^{\alpha})'(x) &\cdots  & (L_{f}^{\alpha})^{(k-1)}(x) }{f\in F_1} \\
    \dosfilas{ L_{f}^{\alpha}(-x) & L_{f}^{\alpha+1}(-x) & \cdots & L_{f}^{\alpha +k-1}(-x) }{f\in F_2}
  \end{array}
  \right| ,
\end{equation}
does not vanish in $[0,+\infty)$. We conjecture that the converse is also true.
We also prove that the polynomials $L_n^{\alpha;\F}$, $n\in \sigma _\F$, are orthogonal with respect to the positive weight
$$
\omega_{\alpha;F} =\frac{x^{\alpha +k}e^{-x}}{(\Omega_F^{\alpha}(x))^2},\quad x>0 .
$$
Moreover, they form a complete orthogonal system in $L^2(\omega _{\alpha;\F})$ (see Theorem \ref{th6.3}).

When $c$ (or $\alpha +1$) and $\F$ are admissible, exceptional Meixner and Laguerre polynomials $m_n^{a,c;\F}$ and $L_n^{\alpha;\F}$, $n\in \sigma _\F$, can be constructed in an alternative way.
Indeed, consider the involution $I$ in the set of all finite sets of positive integers defined by
$$
I(F)=\{1,2,\cdots, f_k\}\setminus \{f_k-f,f\in F\}.
$$
The set $I(F)$ will be denoted by $G$: $G=I(F)$. We also write $G=\{g_1,\cdots , g_m\}$ with $g_i<g_{i+1}$ so that $m$ is the number of elements of $G$ and $g_m$ the maximum element of $G$. We also need the nonnegative integer $v_\F$ defined by
$$
v_\F=u_\F+M_{F_1}+1,
$$
where $M_{F_1}$ is the maximum element of $F_1$.
For the exceptional Meixner polynomials, we then have ($n\ge v_\F$)
\begin{equation}\label{qusmei2i}
m_n^{a,c;\F}(x)=\beta_n
\left|
  \begin{array}{@{}c@{}cccc@{}c@{}}
    &r_0^{\tilde c}(x)m_{n-v_\F}^{a,\tilde c}(x)&r_1^{\tilde c}(x)m_{n-v_\F}^{a,\tilde c}(x-1)&\cdots &r_m^{\tilde c}(x)m_{n-v_\F}^{a,\tilde c}(x-m) & \\
    \dosfilas{ m_{g}^{a,2-\tilde c}(-x-1) & am_{g}^{a,2-\tilde c}(-x) &\cdots  & a^mm_{g}^{a,2-\tilde c}(-x+m-1) }{g\in G_1} \\
    \dosfilas{ m_{g}^{1/a,2-\tilde c}(-x-1) & m_{g}^{1/a,2-\tilde c}(-x) & \cdots & m_{g}^{1/a,2-\tilde c}(-x+m-1)}{g\in G_2}
  \end{array}
  \right|,
\end{equation}
where $\tilde c=c+M_{F_1}+M_{F_2}+2$, $r_j^c(x)=(c+x-m)_{m-j}(x-j+1)_{j}$, $j=0,\ldots ,m$,  and $\beta_n$ is certain normalization constant.

For the exceptional Laguerre polynomials we have ($n\ge v_\F$)
\begin{equation}\label{deflaxai}
L_n^{\alpha ;\F}(x)=\gamma_n
\left|
  \begin{array}{@{}c@{}cccc@{}c@{}}
    &x^mL_{n-v_\F}^{\tilde \alpha}(x)&w_{n-v_\F}^{\tilde \alpha,1} x^{m-1}L_{n-v_\F}^{\tilde \alpha-1}(x)&\cdots &w_{n-v_\F}^{\tilde \alpha,m} L_{n-v_\F}^{\tilde \alpha -m}(x) & \\
    \dosfilas{ L_{g}^{-\tilde \alpha}(-x) & L_{g}^{-\tilde \alpha +1}(-x) &\cdots  & L_{g}^{-\tilde \alpha +m}(-x) }{g\in G_1} \\
    \dosfilas{ L_{g}^{-\tilde \alpha }(x) & (L_{g}^{-\tilde \alpha})'(x) & \cdots & (L_{g}^{-\tilde \alpha})^{(m)}(x)}{g\in G_2}
  \end{array}
  \right|,
\end{equation}
where $\tilde \alpha =\alpha +M_{F_1}+M_{F_2}+2$,  $w^{\alpha ,j} _n=j!\binom{n+\alpha }{j}$, and  $\gamma _n$F  is certain normalization constant.

We have however computational evidence that shows that both identities (\ref{qusmei2i}) and (\ref{deflaxai}) are true for every pair $\F$ of finite sets of positive integers. Using a physics approach, similar formulas to (\ref{deflaxai}) have been introduced by Grandati, Quesne-Grandati and Odake-Sasaki (\cite{G}, \cite{GQ}, \cite{OS3}).

Both determinantal definitions (\ref{defmexi}) and (\ref{qusmei2i}) of the polynomials $m_n^{a,c;\F}$, $n\in \sigma _\F$, automatically imply a couple of factorizations of the second order difference operator $D_\F$ in two first order difference operators. Using these factorizations, we prove that the sequence $m_n^{a,c;\F}$, $n\in \sigma _\F$, and the operator $D_\F$ can be constructed in two different ways using Darboux transforms (see Definition \ref{dxt}). The same happens with de determinantal definitions of the exceptional Laguerre polynomials $L_n^{\alpha, \F}$ (\ref{deflaxi}) and (\ref{deflaxai}). This fact agrees with the G\'omez-Ullate-Kamran-Milson conjecture and its corresponding discrete version (see \cite{GUKM5}): exceptional and exceptional discrete orthogonal polynomials can be obtained by applying a sequence of Darboux transforms to a classical or classical discrete orthogonal family, respectively.

We would like to include in this introduction a conjecture. There seems to be a very nice invariant property of the polynomial $\Omega _\F ^{a,c}$ (\ref{defmexii}) underlying the fact that the polynomials $m_n^{a,c;\F}$, $n\in \sigma _\F$, admit both determinantal definitions (\ref{defmexi}) and (\ref{qusmei2i}):
except for a constant (depending on $a$ but neither on
$x$ nor on $c$), $\Omega_\F^{a,c}(x)$ remains invariant if we
change $\F=(F_1,F_2)$ to $\G=(I(F_1),I(F_2))$, $x$ to $-x$ and
$c$ to $-c-M_{F_1}-M_{F_2}$. More precisely
\begin{equation}\label{iza}
\Omega_\F^{a,c}(x)=(-1)^{u_\F+k_1}\frac{u_a(\F)}{u_a(\G)}\Omega_\G^{a,-c-M_{F_1}-M_{F_2}}
(-x),
\end{equation}
where $u_a(\F)=a^{\binom{k_2}{2}-k_2(k-1)}(1-a)^{k_1k_2}$.

For the cases when $F_1$ is formed by consecutive integers and
$F_2=\emptyset$, or $F_1=\emptyset$ and $F_2$ is formed by
consecutive integers, the conjecture appeared by the first time in
\cite{du2} and it was proved in \cite{du3}.

Passing to the limit, the invariant property (\ref{iza}) gives
\begin{equation}\label{izah}
\Omega _\F^\alpha (x)=\epsilon \Omega _\G^{-\alpha
-M_{F_1}-M_{F_2}-2}(-x),
\end{equation}
where $\epsilon $ is the sign $\epsilon=(-1)^{u_\F+k_1+\sum_{f\in
F_1}f+\sum_{g\in G_1}g}$.

\bigskip
Krawtchouk exceptional polynomials can be formally derived from the Meixner case taking into account that $k_n^{a,N}(x)=m_n^{-a,-N+1}(x)$. That is, by setting $a\to -a$, $c\to -N+1$ in the formulas for the polynomials, and changing
$$
x\in \NN,\quad \frac{a^x\Gamma(x+c)}{x!}\quad \quad \mbox{to\quad  \quad $x=0,\cdots ,N-1,\quad \displaystyle \frac{a^x}{\Gamma(N-x)x!}$}
$$
in the orthogonalizing measure.

\bigskip

We finish pointing out that, as explained above, the approach of this paper is the same as in \cite{duch} for Charlier and Hermite polynomials. Since we work here with a pair of finite sets of positive integers instead of only one set, and more parameters (two for Meixner and one for Laguerre instead of one for Charlier and zero for Hermite), the computations are technically more involve. Anyway, we will omit those proofs which are too similar to the corresponding ones in \cite{duch}.

\section{Preliminaries}
Let $\mu $ be a Borel measure (positive or not) on the real line. The $n$-th moment of $\mu $ is defined by
$\int _\RR t^nd\mu (t)$. When $\mu$ has finite moments for any $n\in \NN$, we can associate it a bilinear form defined in the linear space of polynomials by
\begin{equation}\label{bf}
\langle p, q\rangle =\int pqd\mu.
\end{equation}
Given an infinite set $X$ of nonnegative integers, we say that the polynomials $p_n$, $n\in X$, are orthogonal with respect to $\mu$ if they
are orthogonal with respect to the bilinear form defined by $\mu$; that is, if they satisfy
$$
\int p_np_md\mu =0, \quad n\not = m, \quad n,m \in X.
$$
When $X=\NN$ and the degree of $p_n$ is $n$, $n\ge 0$, we get the usual definition of orthogonal polynomials with respect to a measure.
When $X=\NN$, orthogonal polynomials with respect to a measure are unique up to multiplication by non null constant. Let us remark  that this property is not true when $X\not =\NN$.
Positive measures $\mu $ with finite moments of any order and infinitely many points in its support has always a sequence of orthogonal polynomials $(p_n)_{n\in\NN }$, $p_n$ of degree $n$ (it is enough to apply the Gram-Smith orthogonalizing process to $1, x, x^2, \ldots$); in this case
the orthogonal polynomials have positive norm: $\langle p_n,p_n\rangle>0$. Moreover, given a sequence of orthogonal polynomials $(p_n)_{n\in \NN}$ with respect to a measure $\mu$ (positive or not) the bilinear form (\ref{bf}) can be represented by a positive measure if and only if $\langle p_n,p_n \rangle > 0$, $n\ge 0$.

When $X=\NN$, Favard's Theorem establishes that a sequence $(p_n)_{n\in \NN}$ of polynomials, $p_n$ of degree $n$, is orthogonal (with non null norm) with respect to a measure if and only if it satisfies
a three term recurrence relation of the form ($p_{-1}=0$)
$$
xp_n(x)=a_np_{n+1}(x)+b_np_n(x)+c_np_{n-1}(x), \quad n\ge 0,
$$
where $(a_n)_{n\in \NN}$, $(b_n)_{n\in \NN}$ and $(c_n)_{n\in \NN}$ are sequences of real numbers with $a_{n-1}c_n\not =0$, $n\ge 1$. If, in addition, $a_{n-1}c_n>0$, $n\ge 1$,
then the polynomials $(p_n)_{n\in \NN}$ are orthogonal with respect to a positive measure with infinitely many points in its support, and conversely.
Again, Favard's Theorem is not true for a sequence of orthogonal polynomials $(p_n)_{n\in X}$ when $X\not =\NN$.

Darboux transformations are an important tool for constructing exceptional orthogonal polynomials. We define them next for second order difference and differential operators.

\begin{definition}\label{dxt}
Given a system $(T,(\phi_n)_n)$ formed by a second order difference or differential operator $T$ and a sequence $(\phi_n)_n$ of eigenfunctions for $T$, $T(\phi_n)=\pi_n\phi_n$, by a Darboux transform of the system $(T,(\phi_n)_n)$ we
mean the following. For a real number $\lambda$, we factorize $T-\lambda Id$ as the product of two first order difference or differential operators $T=BA+\lambda Id$ ($Id$ denotes the identity operator). We then produce a new system consisting in the operator $\hat T$, obtained by reversing the order of the factors,
$\hat T = AB+\lambda Id$, and the sequence of eigenfunctions $\hat \phi_n =A(\phi_n)$: $\hat T(\hat \phi_n)=\pi_n\hat\phi_n$.
We say that the system $(\hat T,(\hat\phi_n)_n)$ has been obtained by applying a Darboux transformation with parameter $\lambda$ to
 the system $(T,(\phi_n)_n)$.
\end{definition}

We will also need the following straightforward lemma.

\bigskip
\begin{lemma}\label{rmc}
Let $M$ be a $(s+1)\times m$ matrix with $m\ge s+1$. Write $c_i$, $i=1,\ldots , m$, for the columns of $M$ (from left to right). Assume that for $0\le j\le m-s-1$ the consecutive columns $c_{j+i}$, $i=1,\cdots ,s$, of $M$ are linearly independent while the consecutive columns $c_{j+i}$, $i=1,\cdots ,s+1$, are linearly dependent. Then $\rank M=s$.
\end{lemma}

\bigskip

Given a finite set of numbers $F=\{f_1,\cdots, f_k\}$ we denote by $V_F$ the Vandermonde determinant defined by
\begin{align}\label{defvdm}
V_F=\prod_{1=i<j=k}(f_j-f_i).
\end{align}

\subsection{Christoffel transform}\label{secChr}
Let $\mu$ be a measure (positive or not) and assume that $\mu$ has a sequence of orthogonal polynomials
$(p_n)_{n\in \NN}$, $p_n$ with degree $n$ and $\langle p_n,p_n\rangle \not =0$ (as we mentioned above, that always happens if $\mu$ is positive, with finite moments and infinitely many points in its support).

Given a finite set $F$ of real numbers, $F=\{f_1,\cdots , f_k\}$, $f_i<f_{i+1}$, we write $\Phi_n$, $n\ge 0$, for the $k\times k$ determinant
\begin{equation}\label{defph}
\Phi_n=\vert p_{n+j-1}(f_i)\vert _{i,j=1,\cdots , k}.
\end{equation}
Notice that $\Phi_n$, $n\ge 0$, depends on both, the finite set $F$ and the measure $\mu$. In order to stress this dependence, we sometimes write in this Section $\Phi_n^{\mu, F}$ for $\Phi_n$.

Along this Section we assume that the set $\X_\mu^F=\{ n\in \NN :\Phi_n^{\mu,F}=0\}$ is finite. We denote $\x_\mu ^F=\max \X_\mu ^F$. If
$\X_\mu ^F=\emptyset$ we take $\x_\mu ^F=-1$.

The Christoffel transform of $\mu$ associated to the annihilator polynomial $\pp$ of $F$,
$$
\pp (x)=(x-f_1)\cdots (x-f_k),
$$
is the measure defined by $ \mu_F =\pp \mu$.

Orthogonal polynomials with respect to $\mu_F$ can be constructed by means of the formula
\begin{equation}\label{mata00}
q_n(x)=\frac{1}{\pp (x)}\det \begin{pmatrix}p_n(x)&p_{n+1}(x)&\cdots &p_{n+k}(x)\\
p_n(f_1)&p_{n+1}(f_1)&\cdots &p_{n+k}(f_1)\\
\vdots&\vdots&\ddots &\vdots\\
p_n(f_k)&p_{n+1}(f_k)&\cdots &p_{n+k}(f_k) \end{pmatrix}.
\end{equation}
Notice that the degree of $q_n$ is equal to $n$ if and only if $n\not\in \X_\mu ^F$. In that case the leading coefficient $\lambda^Q_n$ of $q_n$ is
equal to $(-1)^k\lambda^P_{n+k}\Phi_n$, where $\lambda ^P_n$ denotes the leading coefficient of $p_n$.

The next Lemma follows easily using \cite{Sz}, Th. 2.5.

\begin{lemma}\label{sze}
The measure $\mu_F$ has a sequence $(q_n)_{n=0}^\infty $, $q_n$ of degree $n$, of orthogonal polynomials if and only if $\X_\mu ^F=\emptyset$.
In that case, an orthogonal polynomial of degree $n$ with respect to $\mu _F$ is given by (\ref{mata00}) and also $\langle q_n,q_n\rangle _{\mu _F}\not =0$, $n\ge 0$. If $\X_\mu \not =\emptyset$, the polynomial $q_n$ (\ref{mata00}) has still degree $n$ for $n\not \in \X_\mu^F$, and satisfies $\langle q_n,r\rangle_{\mu _F}=0$ for all polynomial $r$ with degree less than $n$ and $\langle q_n,q_n\rangle _{\mu _F}\not =0$.
\end{lemma}

From (\ref{mata00}), one can also deduce (see Lemma 2.8 of \cite{duch})
\begin{equation}\label{n2q}
\langle q_n,q_n\rangle _{\mu_F}=(-1)^k\frac{\lambda^P_{n+k}}{\lambda^P_{n}}\Phi_n\Phi_{n+1}\langle p_n,p_n\rangle _{\mu},\quad n> \x_\mu^F+1.
\end{equation}
This identity holds for $n\ge 0$ when $\X_\mu =\emptyset$

\subsection{Finite sets and pair of finite sets of positive integers.}\label{sfspi}
For a finite set $F$ of positive integers, we denote $M_F=\max F$,
$m_F=\min F$; if $F=\emptyset$, we define $M_F=m_F=-1$.

Consider the set $\Upsilon$  formed by all finite sets of positive
integers:
\begin{align*}
\Upsilon=\{F:\mbox{$F$ is a finite set of positive integers}\} .
\end{align*}
We consider the involution $I$ in $\Upsilon$ defined by
\begin{align}\label{dinv}
I(F)=\{1,2,\cdots, M_F\}\setminus \{M_F-f,f\in F\}.
\end{align}
The definition of $I$ implies that $I^2=Id$.

The set $I(F)$ will be denoted by $G$: $G=I(F)$. Notice that
$$
M_F=M_G,\quad m=M_F-k+1,
$$
where $k$ and $m$ are the number of elements of $F$ and $G$,
respectively.

For a finite set $F=\{f_1,\cdots ,f_k\}$, $f_i<f_{i+1}$, of
positive integers, we define the number $s_F$ by
\begin{equation}\label{defs0}
s_F=\begin{cases} 1,& \mbox{if $F=\emptyset$},\\
k+1,&\mbox{if $F=\{1,2,\cdots , k\}$},\\
\min \{s\ge 1:s<f_s\}, & \mbox{if $F\not =\{1,2,\cdots k\}$},
\end{cases}
\end{equation}
and the set $F_{\Downarrow}$ of positive integers by
\begin{equation}\label{defffd}
F_{\Downarrow}=\begin{cases} \emptyset,& \mbox{if $F=\emptyset$ or $F=\{1,2,\cdots , k\}$,}\\
\{f_{s_F}-s_F,\cdots , f_k-s_F\},& \mbox{if $F\not =\{1,2,\cdots ,
k\}$}.
\end{cases}
\end{equation}

From now on, $\F=(F_1,F_2)$ will denote a pair of finite sets of
positive integers. We will write $F_1=\{ f_1^1,\cdots ,
f_{k_1}^1\}$, $F_2=\{ f_1^2,\cdots , f_{k_2}^2\}$, with
$f_i^j<f_{i+1}^j$. Hence $k_j$ is the number of elements of $F_j$,
$j=1,2$, and $k=k_1+k_2$ is the number of elements of $\F$. One of
the components of $\F$, but not both, can be the empty set.

We associate to $\F$ the nonnegative integers $u_\F$ and $v_\F$ and the infinite set of nonnegative integers $\sigma_\F$ defined by
\begin{align}\label{defuf}
u_\F&=\sum_{f\in F_1}f+\sum_{f\in
F_2}f-\binom{k_1+1}{2}-\binom{k_2}{2},\\\label{defvf}
v_\F&=\sum_{f\in F_1}f+\sum_{f\in
F_2}f-\binom{k_2}{2}+M_{F_1}-\frac{(k_1-1)(k_1+2)}{2},\\\label{defsf}
\sigma _\F&=\{u_\F,u_\F+1,u_\F+2,\cdots \}\setminus \{u_\F+f,f\in
F_1\}.
\end{align}
The infinite set $\sigma_\F$ will be the set of indices for the exceptional Meixner or Laguerre polynomials associated to $\F$.

Notice that $v_\F=u_\F+M_{F_1}+1$; hence
$\{v_\F,v_\F+1,v_\F+2,\cdots \}\subset \sigma_\F$.

For a pair $\F=(F_1,F_2)$ of positive integers we denote by
$\F_{j,\{ i\}}$, $i=1,\ldots , k_j$, $j=1,2$, and
$\F_{\Downarrow}$ the pair of finite sets of positive integers
defined by
\begin{align}\label{deff1}
\F_{1,\{ i\} }&=(F_1\setminus \{f_i^1\},F_2),\\\label{deff2}
\F_{2,\{ i\} }&=(F_1,F_2\setminus \{f_i^2\}),\\\label{deffd}
\F_{\Downarrow}&=((F_1)_{\Downarrow},F_2),
\end{align}
where $(F_1)_{\Downarrow}$ is defined by (\ref{defffd}). We also define
\begin{equation}\label{defs0f}
s_\F=s_{F_1}
\end{equation}
where the number $s_{F_1}$ is defined by (\ref{defs0}).

\subsection{Admissibility}\label{sectadm}
Using the determinants (\ref{defmexi}) and (\ref{deflaxi}), whose entries are Meixner $m_n^{a,c}$ or Laguerre polynomials $L_n^\alpha$, respectively, we will associate to each pair $\F$ of finite sets of positive integers a sequence of polynomials which are  eigenfunctions of a second order difference or differential operator, respectively.
The more important of these examples are those which, in addition, are orthogonal and complete with respect to a positive measure. The key concept for the existence of such positive measure is that of admissibility.

Let us remind that for Charlier and Hermite polynomials, the admissibility of a finite set $F$ of positive integers is defined as follows.

\begin{definition}\label{dadmch}
Let $F$ be a finite set of positive integers. Split up the set $F$, $F=\bigcup _{i=1}^KY_i$, in such a way that $Y_i\cap Y_j=\emptyset $, $i\not =j$, the elements of each $Y_i$ are consecutive integers and $1+\max (Y_i)<\min Y_{i+1}$, $i=1,\cdots, K-1$. We then say that $F$ is admissible if each $Y_i$, $i=1,\cdots, K$,  has an even number of elements.
\end{definition}

It is easy to see that this is equivalent to
$$
\prod_{f\in F}(x-f)\ge 0, \quad x\in \NN.
$$
This implies that the measure $\rho_a^F=\sum_{x=0}^\infty \prod_{f\in F}(x-f)a^x/x!\delta_x$ is positive. As shown in \cite{duch}, Charlier exceptional polynomials are dual of the orthogonal polynomials with respect to this measure.

Given a pair $\F$ of finite sets of positive integers, consider the measure $\rho_{a,c}^\F$ defined by (\ref{ctmew}). We show in the next Section that Meixner exceptional polynomials are dual of the orthogonal polynomials with respect to this measure. Hence, the admissibility condition in the Meixner case should be equivalent to the positivity of the measure $\rho_{a,c}^\F$.

\begin{definition}\label{dadm} Let $\F=(F_1,F_2)$ be  a pair of finite sets of positive integers. For a real number $c\not=0,-1,-2,\cdots$, write
$\hat c=\max \{-[c],0\}$, where $[c]$ denotes the integral part of $c$. We say that $c$ and $\F$ are admissible if for all $x\in \NN $
\begin{equation}\label{defadm}
\frac{\prod_{f\in F_1}(x-f)\prod_{f\in F_2}(x+c+f)}{(x+c)_{\hat c}}\ge 0.
\end{equation}
\end{definition}

As we wrote in the introduction, this admissibility concept is more involve than the corresponding for exceptional Charlier and Hermite polynomials. Indeed, on the one hand, we have now a pair $\F$ of finite sets instead of a single finite set $F$. On the other hand, the admissibility also depends on the parameter $c$ of the Meixner polynomials (or on the parameter $\alpha$ of the Laguerre polynomials).

In the following Lemma we include some important consequences derived from the admissibility.

\begin{lemma}\label{ladm} Given  a real number $c\not =0,-1,-2,\ldots$, and a pair $\F $ of finite sets of positive integers, we have
\begin{enumerate}
\item if $c$ and $\F$ are admissible then $c+k>0$.
\item If $c>0$, then $c$ and $\F$ are admissible if and only if $F_1$ is admissible (in the sense of Definition \ref{dadmch}).
\item If $F_1=\emptyset$, $c$ and $\F$ are admissible if and only if $c>0$.
\item If $c$ and $\F$ are admissible then $c+s_\F$ and $\F _\Downarrow$ are admissible (where the positive integer $s_\F$ and the pair $\F _\Downarrow$ are defined by (\ref{defs0f}) and (\ref{deffd}), respectively).
\end{enumerate}
\end{lemma}

\begin{proof}

\noindent
Proof of (1). We first point out that
\begin{equation}\label{xcN}
\sign ((x+c)_{\hat c})=\begin{cases} (-1)^{{\hat c}-x},& 0\le x\le {\hat c},\\
1,& x>{\hat c}. \end{cases}
\end{equation}
Given an $l$-tuple $A=(a_1,\ldots, a_l)$ of non null real numbers, we denote by $n_{\pm}(A)$ the number of sign changes along the elements of $A$ (for instance, if $A=(-\pi,2,1,-\sqrt 2,-1,1,1)$ then $n_{\pm}(A)=3$).

We next prove that given a finite set $I$ of nonnegative integers with elements ordered in increasing size, we have
\begin{equation}\label{dnpm}
n_{\pm}((\prod_{f\in F_2}(x+c+f),x\in I))\le |F_2\cap ({\hat c}-I)|,
\end{equation}
where $|X|$ denotes the number of elements of the finite set $X$ and $\hat c-I$ denotes the set $\{\hat c-i:i\in I\}$.
Indeed, for $a\in I$, write $A_a=\{ f\in F_2:a+c+f<0\}$. Notice that $\prod_{f\in F_2}(a+c+f)$ is positive or negative depending on whether $|A_a|$ is even or odd, respectively. Take now consecutive numbers $x=a, x=a+1\in I$ where $\prod_{f\in F_2}(x+c+f)$ changes its sign.
Then $A_{a+1}\not = A_a$ because $|A_{a+1}|$ and $|A_{a}|$ has different parity. Since $A_{a+1}\subset A_a$, there exists $f_a\in A_a\setminus A_{a+1}$.
That is, $a+c+f_a<0<a+c+1+f_a$, or $-a-f_a-1<c<-a-f_a$. In other words $-{\hat c}=-a-f_a$, and then $f_a\in {\hat c}-I$.

Take now $x=a,x=b\in I$, with $a+1<b$, and $i\not \in I$ if $a<i<b$, and where $\prod_{f\in F_2}(x+c+f)$ changes its sign. Abusing of notation we write $f_a=\max A_a$. If $a+c+f_a+1>0$, proceeding as before we get $f_a\in {\hat c}-I$ and $f_a\not \in A_b$. On the other hand, if $a+c+f_a+1<0$, by definition of $f_a$, we conclude that $f_a+1\not \in F_2$, and then $\prod_{f\in F_2}(x+c+f)$ does not change its sign from $a$ to $a+1$. In the same way, we have that if $a+c+f_a+2>0$ then $f_a\in {\hat c}-I$ and $f_a\not \in A_b$, while if $a+c+f_a+2<0$ then $\prod_{f\in F_2}(x+c+f)$ does not change its sign from $a$ to $a+2$; in particular $b>a+2$. Since $\prod_{f\in F_2}(x+c+f)$ changes its sign from $a$ to $b$ proceeding in the same way, we can conclude that $f_a\in {\hat c}-I$ and $f_a\not \in A_b$.

We have then proved that if $\prod_{f\in F_2}(x+c+f)$ changes its sign in two consecutive elements $a$ and $b$ of $I$ then there exists $f_a\in F_2$ satisfying $f_a\in {\hat c}-I$, $f_a\in A_a$ and $f_a\not \in A_b$. This implies that (\ref{dnpm}) holds.

Decompose now the set $\{0,1,\ldots , {\hat c}\}$ as follows:
$$
\{0,1,\ldots , {\hat c}\} = X_1\cup Y_1\cup X_2\cup Y_2\cup \ldots \cup X_l\cup Y_l,
$$
where each $X_i$, $Y_i$ is formed by consecutive nonnegative integers, $1+\max X_i=\min Y_i$, $1+\max Y_i=\min X_{i+1}$, $X_i\cap F_1=\emptyset$, $Y_i\subset F_1$ and $Y_l=\emptyset $ if $f^1_{k_1}<{\hat c}$ (let us remind that $k_1$ is the number of elements of $F_1$ and
that $f^1_{k_1}$ is the maximum element of $F_1$). Write $x_i=|X_i|$, $y_i=|Y_i|$, $i=1,\ldots , l$. Since $X_i, Y_i$, $i=1,\ldots , l$, are disjoint sets and $Y_i\subset F_1$, we get
\begin{equation}\label{ssk2}
{\hat c}+1=\sum_{i=1}^l(x_i+y_i),\quad \sum_{i=1}^ly_i\le k_1.
\end{equation}
Notice that the sign of $\prod_{f\in F_1}(x-f)$ is constant in each $X_i$. Since $(x+c)_{\hat c}$ alternates its sign in consecutive numbers
of $\{0,1,\ldots \hat c\}$ (see (\ref{xcN})), (\ref{defadm}) implies that $\prod_{f\in F_2}(x+c+f)$ changes its sign $(x_i-1)$-times in $X_i$. On the other hand, one can carefully check that also $\prod_{f\in F_2}(x+c+f)$ changes its sign between the maximum element of $X_i$ and the minimum element of $X_{i+1}$. That means that $\prod_{f\in F_2}(x+c+f)$ changes its sign $(-1+\sum_{i=1}^lx_i)$-times in $\cup_{i=1}^lX_i$. Hence, (\ref{dnpm}) gives
$$
\sum_{i=1}^lx_i\le |F_2\cap({\hat c}-\cup_{i=1}^lX_i)|+1\le k_2 +1.
$$
(\ref{ssk2}) gives then ${\hat c}+1\le k_1+k_2+1$. From where we get $c+k>0$.

\bigskip

\noindent
Proof of (2). It is a straightforward consequence of the following fact: if $c>0$ then
$$
\sign \left(\frac{\prod_{f\in F_1}(x-f)\prod_{f\in F_2}(x+c+f)}{(x+c)_{\hat c}}\right)=\sign \prod_{f\in F_1}(x-f).
$$

\bigskip

\noindent
Proof of (3). Assume $F_1=\emptyset$, which it is an admissible set (in the sense of Definition \ref{dadmch}). If $c>0$, using (2) of this Lemma we deduce that $c$ and $\F$ are admissible. On the other hand, if $c$ and $\F$ are admissible we have ($F_1=\emptyset$ implies that $\prod_{f\in F_1}(x-f)=1$, $x\ge 0$) from (\ref{defadm}) that
$$
\sign \prod_{f\in F_2}(x+c+f)=\sign (x+c)_{\hat c},\quad x\ge 0.
$$
If $c<0$, then ${\hat c}>0$. Take $x_{\hat c}={\hat c}-1\ge 0$. Hence $\sign (x_{\hat c}+c)_{\hat c}=-1$ and so
 $\sign \prod_{f\in F_2}(x_{\hat c}+c+f)=-1$. Since $-{\hat c}-1<c<{\hat c}$, we have $-1<x_{\hat c}+c$ and then
$$
\{f\in F_2: x_{\hat c}+c+f<0\} \subset \{f\in F_2: -1+f<0\}=\emptyset .
$$
Hence $\sign \prod_{f\in F_2}(x_{\hat c}+c+f)=1$. Which it is a contradiction.

\bigskip

\noindent
Proof of (4). Since $\F _\Downarrow=\{ (F_1) _\Downarrow, F_2\}$, we have to prove that
$$
H(x)=\frac{\prod_{f\in (F_1) _\Downarrow}(x-f)\prod_{f\in F_2}(x+c+s_\F+f)}{(x+c+s_\F)_{{\hat c}+s_\F}}\ge 0,\quad x\ge 0.
$$
Using the definition of $(F_1) _{\Downarrow}$ (\ref{defffd}), we have for $x\ge 0$ that
$$
H(x)=\frac{H(x+s_\F)}{\prod_{j=1}^{s_\F}(x+s_\F-j)(x+s_\F+c+{\hat c})_{s_\F}}\ge 0 .
$$

\end{proof}

\subsection{Meixner and Laguerre polynomials}
We include here basic definitions and facts about Meixner and Laguerre polynomials, which we will need in the following Sections.

For $a\not =0, 1$ we write $(m_{n}^{a,c})_n$ for the sequence of Meixner polynomials defined by
\begin{equation}\label{Mxpol}
m_{n}^{a,c}(x)=\frac{a^n}{(1-a)^n}\sum _{j=0}^n a^{-j}\binom{x}{j}\binom{-x-c}{n-j}
\end{equation}
(we have taken a slightly different normalization from the one used in \cite{Ch}, pp. 175-7, from where
the next formulas can be easily derived; see also \cite{KLS}, pp, 234-7 or \cite{NSU}, ch. 2).
Meixner polynomials are eigenfunctions of the following second order difference operator
\begin{equation*}\label{Mxdeq}
D_{a,c} =\frac{x\Sh_{-1}-[(1+a)x+ac]\Sh_0+a(x+c)\Sh_1}{a-1},\qquad D_{a,c} (m_{n}^{a,c})=nm_{n}^{a,c},\quad n\ge 0.
\end{equation*}
When $a\not =0, 1$, they satisfy the following three term recurrence formula ($m_{-1}=0$)
\begin{equation}\label{Mxttrr}
xm_n=(n+1)m_{n+1}-\frac{(a+1)n+ac}{a-1}m_n+\frac{a(n+c-1)}{(a-1)^2}m_{n-1}, \quad n\ge 0
\end{equation}
(to simplify the notation we remove the parameters in some formulas).
Hence, for $a\not =0,1$ and $c\not =0,-1,-2,\ldots $, they are always orthogonal with respect to a moment functional $\rho_{a,c}$, which we
normalize
by taking $\langle \rho_{a,c},1\rangle =\Gamma(c)/(1-a)^c$. For $0<\vert a\vert<1$ and $c\not =0,-1,-2,\ldots $, we have
\begin{equation*}\label{MXw}
\rho_{a,c}=\sum _{x=0}^\infty \frac{a^{x}\Gamma(x+c)}{x!}\delta _x,
\end{equation*}
and
\begin{equation}\label{norme}
\langle m_n^{a,c},m_n^{a,c}\rangle =\frac{a^{n}\Gamma(n+c)}{n!(1-a)^{2n+c}}.
\end{equation}
The moment functional $\rho_{a,c}$ can be represented by a positive measure only when $0<a<1$ and $c>0$.

Meixner polynomials satisfy the following identities ($\quad n,m\in \NN, x\in \RR$)
\begin{align}\label{sdm}
m_{n}^{a,c}(x+1)-m_{n}^{a,c}(x)&=m_{n-1}^{a,c+1}(x),\\ \label{sdm2}
m_{n}^{1/a,c}(x+1)-am_{n}^{1/a,c}(x)&=(1-a)m_{n}^{1/a,c+1}(x),\\
\label{sdm2b}
a^{m-n}n!(1+c)_{m-1}m_{n}^{a,c}(m)&=(a-1)^{m-n}m!(1+c)_{n-1}m_{m}^{a,c}(n),\\
\label{sdm3}
m_{n}^{a,c}(x)&=(-1)^nm_{n}^{1/a,c}(-x-c).
\end{align}

For $\alpha\in\mathbb{R}$, we write $(L_n^\alpha )_n$ for the sequence of Laguerre polynomials
\begin{equation}\label{deflap}
L_n^{\alpha}(x)=\sum_{j=0}^n\frac{(-x)^j}{j!}\binom{n+\alpha}{n-j}
\end{equation}
(that and the next formulas can be found in \cite{EMOT}, vol. II, pp. 188--192; see also \cite{KLS}, pp, 241-244).

They satisfy the three-term recurrence formula ($L_{-1}^{\alpha}=0$)
\begin{equation*}\label{ttrrL}
xL_n^{\alpha}=-(n+1)L_{n+1}^{\alpha}+(2n+\alpha+1)L_n^{\alpha}-(n+\alpha)L_{n-1}^{\alpha}.
\end{equation*}
Hence, for $\alpha\neq-1,-2,\ldots$, they are orthogonal with respect to a measure $\mu_{\alpha}=\mu_{\alpha}(x)dx$. This measure is positive
only when $\alpha>-1$ and then
$$
\mu_{\alpha}(x) =x^\alpha e^{-x}, x>0.
$$
The Laguerre polynomials are eigenfunctions of the following second-order differential operator
\begin{equation}\label{DopL}
D_{\alpha}=-x\left(\frac{d}{dx}\right)^2-(\alpha+1-x)\frac{d}{dx},\quad D_{\alpha}(L_n^{\alpha})=nL_n^{\alpha},\quad n\geq0.
\end{equation}
We will also use the following formulas
\begin{equation}\label{Lagder}
    \left(L_n^{\alpha}\right)'=-L_{n-1}^{\alpha+1},
\end{equation}
\begin{equation}\label{Lagab}
   L_n^{\alpha}=L_{n-1}^{\alpha}+L_n^{\alpha-1}.
\end{equation}
One can obtain Laguerre polynomials from Meixner polynomials using the limit
\begin{equation}\label{blmel}
\lim_{a\to 1}(a-1)^nm_n^{a,c}\left(\frac{x}{1-a}\right)=L_n^{c-1}(x)
\end{equation}
see \cite{KLS}, p. 243 (take into account that we are using for the
Meixner polynomials a different normalization to that in \cite{KLS}). The previous limit is uniform in compact sets of $\CC$.

\section{Constructing polynomials which are eigenfunctions of second order difference operators}
As in Section \ref{sfspi}, $\F=(F_1,F_2)$ will denote a pair of finite sets of positive integers. We will write $F_i=\{ f_1^i,\cdots , f_{k_i}^i\}$, with $f_j^i<f_{j+1}^i$, $i=1,2$. Hence $k_i$ is the number of elements of $F_i$ and $f_{k_i}$ is the maximum element of $F_i$.

For real numbers $a,c$, with $a\not = 0,1$ and $c\not =0,-1,-2,\ldots$, we associate to each pair $\F$ the polynomials $m_n^{a,c;\F}$, $n\in \sigma_\F$, displayed in the following definition.
It turns out that these polynomials are always eigenfunctions of a second order difference operator with rational coefficients. We call them exceptional Meixner polynomials when, in addition, they are orthogonal and complete with respect to a positive measure (this will happen as long as $c$ and the pair $\F$ are admissible; see Definition \ref{defadm} in the previous Section).

\begin{definition}
Let $\F =(F_1,F_2)$ be a pair of finite sets of positive integers. We define the polynomials $m_n^{a,c;\F}$, $n\in \sigma _\F$, as
\begin{equation}\label{defmex}
m_n^{a,c;\F}(x)=  \left|
  \begin{array}{@{}c@{}cccc@{}c@{}}
    & m_{n-u_\F}^{a,c}(x)&m_{n-u_\F}^{a,c}(x+1)&\cdots &m_{n-u_\F}^{a,c}(x+k) & \\
    \dosfilas{ m_{f}^{a,c}(x) & m_{f}^{a,c}(x+1) &\cdots  & m_{f}^{a,c}(x+k) }{f\in F_1} \\
    \dosfilas{ m_{f}^{1/a,c}(x) & m_{f}^{1/a,c}(x+1)/a & \cdots & m_{f}^{1/a,c}(x+k)/a^k }{f\in F_2}
  \end{array}
  \right|
\end{equation}
where the number $u_\F$ and the infinite set of nonnegative integers $\sigma _\F$ are defined by (\ref{defuf}) and (\ref{defsf}), respectively.
\end{definition}
The determinant (\ref{defmex}) should be understood as explained in the Introduction (see (\ref{defdosf})).

To simplify the notation, we will sometimes write $m_n^\F=m_n^{a,c;\F}$.

Using Lemma 3.4 of \cite{DdI}, we deduce that $m_n^\F$, $n\in \sigma _\F$, is a polynomial of degree $n$ with leading coefficient equal to
\begin{equation}\label{lcrn}
(-1)^{k_2(k_1+1)} \frac{(a-1)^{k_2(k_1+1)}V_{F_1}V_{F_2}\prod_{f\in F_1}(f-n+u_\F)}{a^{k_2k_1+\binom{k_2+1}{2}}(n-u_\F)!\prod_{f\in F_1}f!\prod_{f\in F_2}f!},
\end{equation}
where $V_F$ is the Vandermonde determinant (\ref{defvdm}). With the convention that $m_n^{a,c}=0$ for $n<0$, the determinant (\ref{defmex}) defines a polynomial for any $n\ge 0$, but for $n\not \in \sigma_\F$ we have $m_n^\F=0$.

Combining columns in (\ref{defmex}) and taking into account (\ref{sdm}) and (\ref{sdm2}), we have the alternative definition
\begin{equation}\label{defmexa}
m_n^{a,c;\F}(x)= \left|
  \begin{array}{@{}c@{}cccc@{}c@{}}
    & m_{n-u_\F}^{a,c}(x)&m_{n-u_\F-1}^{a,c+1}(x)&\cdots &m_{n-u_\F-k}^{a,c+k}(x) & \\
    \dosfilas{ m_{f}^{a,c}(x) & m_{f-1}^{a,c+1}(x) &\cdots  & m_{f-k}^{a,c+k}(x) }{f\in F_1} \\
    \dosfilas{ m_{f}^{1/a,c}(x) & \frac{1-a}{a}m_{f}^{1/a,c+1}(x) & \cdots & \frac{(1-a)^k}{a^k}m_{f}^{1/a,c+k}(x)}{f\in F_2}
  \end{array}
  \right| .
\end{equation}

The  polynomials $m_n^\F$, $n\in \sigma_\F$, are strongly related by duality with the polynomials $q_n^\F$, $n\ge 0$, defined by
\begin{equation}\label{defqnme}
q_n^\F(x)=\frac{\left|
  \begin{array}{@{}c@{}cccc@{}c@{}}
    & m_n^{a,c}(x-u_\F)&m_{n+1}^{a,c}(x-u_\F)&\cdots &m_{n+k}^{a,c}(x-u_\F) & \\
    \dosfilas{ m_{n}^{a,c}(f) & m_{n+1}^{a,c}(f) &\cdots  & m_{n+k}^{a,c}(f) }{f\in F_1} \\
    \dosfilas{m_{n}^{1/a,c}(f) & -m_{n+1}^{1/a,c}(f) & \cdots & (-1)^{k}m_{n+k}^{1/a,c}(f)}{f\in F_2}
  \end{array}
  \right|}{(-1)^{nk_2}\prod_{f\in F_1}(x-f-u_\F)\prod_{f\in F_2}(x+c+f-u_\F)}.
\end{equation}

\begin{lemma}\label{lem3.2}
If $u$ is a nonnegative integer and $v\in \sigma_\F$, then
\begin{equation}\label{duaqnrn}
q_u^\F(v)=\kappa\xi_u\zeta_vm_v^\F(u),
\end{equation}
where
\begin{align*}
\kappa&=\frac{(-1)^{\sum _{f\in F_2}f}a^{k_2(k_1+1)+\sum _{f\in F_2}f}\prod_{f\in F_1}f!\prod_{f\in F_2}f!}{(a-1)^{k_2(k_1+1)}\prod_{f\in F_1}(1+c)_{f-1}\prod_{f\in F_2}(1+c)_{f-1}},\\
\xi_u&=\frac{a^{(k_1+1)u}\prod_{i=0}^k(1+c)_{u+i-1}}
{(a-1)^{(k+1)u}\prod_{i=0}^k(u+i)!},\\
\zeta_v&=\frac{(a-1)^{v}(v-u_\F)!}{a^v(1+c)_{v-u_\F-1}\prod_{f\in F_1}(v-f-u_\F)\prod_{f\in F_2}(v+c+f-u_\F)}.
\end{align*}
\end{lemma}

\begin{proof}
It is a straightforward consequence of the duality (\ref{sdm2b}) for the Meixner polynomials.

\end{proof}

We now prove that the polynomials $m_n^\F$, $n\in \sigma_\F$, are eigenfunctions of a second order difference operator with rational coefficients. To establish the result in full, we need some more notations.
We denote by $\Omega _\F ^{a,c}(x)$ and $\Lambda _\F^{a,c}(x)$ the polynomials
\begin{align}\label{defom}
\Omega _\F^{a,c}(x)&=  \left|
  \begin{array}{@{}c@{}cccc@{}c@{}}
    \dosfilas{ m_{f}^{a,c}(x) & m_{f}^{a,c}(x+1) &\cdots  & m_{f}^{a,c}(x+k-1) }{f\in F_1} \\
    \dosfilas{ m_{f}^{1/a,c}(x) & m_{f}^{1/a,c}(x+1)/a & \cdots & m_{f}^{1/a,c}(x+k-1)/a^{k-1} }{f\in F_2}
  \end{array}
  \right|,\\
\label{deflam}
\Lambda _\F^{a,c}(x)&=\left|
  \begin{array}{@{}c@{}ccccc@{}c@{}}
    \dosfilas{ m_{f}^{a,c}(x) & m_{f}^{a,c}(x+1) &\cdots  & m_{f}^{a,c}(x+k-2)& m_{f}^{a,c}(x+k) }{f\in F_1} \\
    \dosfilas{ m_{f}^{1/a,c}(x) & \displaystyle \frac{m_{f}^{1/a,c}(x+1)}{a} & \cdots & \displaystyle \frac{m_{f}^{1/a,c}(x+k-2)}{a^{k-2}} & \displaystyle \frac{m_{f}^{1/a,c}(x+k)}{a^{k}} }{f\in F_2}
  \end{array}
  \right|.
\end{align}
To simplify the notation we sometimes write $\Omega_\F=\Omega_\F^{a,c}$, $\Lambda_\F^{a,c}=\Lambda_\F^{a,c}$.
Using Lemma 3.4 of \cite{DdI}, we deduce that the degree of both $\Omega _\F$ and $\Lambda_\F$
is $u_\F+k_1$. Moreover, the leading coefficient of $\Omega _\F$ is
$$
\frac{V_{F_1}V_{F_2}a^{\binom{k_2}{2}-k_2(k-1)}(a-1)^{k_1k_2}}{\prod_{f\in F_1}f!\prod_{f\in F_2}f!}.
$$

As for $m_n^\F$ (see (\ref{defmexa})), we have for $\Omega_\F$ the following alternative definition
\begin{equation}\label{defoma}
\Omega _\F(x)=\left|
  \begin{array}{@{}c@{}cccc@{}c@{}}
    \dosfilas{ m_{f}^{a,c}(x) & m_{f-1}^{a,c+1}(x) &\cdots  & m_{f-k+1}^{a,c+k-1}(x) }{f\in F_1} \\
    \dosfilas{ m_{f}^{1/a,c}(x) & \frac{1-a}{a}m_{f}^{1/a,c+1}(x) & \cdots & \frac{(1-a)^k}{a^k}m_{f}^{1/a,c+k}(x)}{f\in F_2}
  \end{array}
  \right|.
\end{equation}

From here and (\ref{defmexa}), it is easy to deduce that
\begin{equation}\label{rrom}
m_{u_\F}^{a,c;\F}(x)=\left(\frac{1-a}{a}\right)^{s_\F k_2}\Omega^{a,c+s_0}_{\F_\Downarrow }(x),
\end{equation}
where the positive integer $s_\F$ and the pair $\F_\Downarrow$ of finite sets of positive integers are defined by (\ref{defs0f}) and
(\ref{deffd}), respectively.

We also need the determinants $\Phi_n^\F$ and $\Psi_n^\F$, $n\ge 0$, defined by
\begin{align}\label{defphme}
\Phi^\F_n&=(-1)^{nk_2}\left|
  \begin{array}{@{}c@{}cccc@{}c@{}}
        \dosfilas{ m_{n}^{a,c}(f) & m_{n+1}^{a,c}(f) &\cdots  & m_{n+k-1}^{a,c}(f) }{f\in F_1} \\
    \dosfilas{ m_{n}^{1/a,c}(f) & -m_{n+1}^{1/a,c}(f) & \cdots & (-1)^{k-1}m_{n+k-1}^{1/a,c}(f)}{f\in F_2}
  \end{array}
  \right|,\\\label{defpsme}
\Psi_n^F&=(-1)^{nk_2}\left|
  \begin{array}{@{}c@{}ccccc@{}c@{}}
        \dosfilas{ m_{n}^{a,c}(f) & m_{n+1}^{a,c}(f) &\cdots  & m_{n+k-2}^{a,c}(f)& m_{n+k}^{a,c}(f) }{f\in F_1} \\
    \dosfilas{ m_{n}^{1/a,c}(f) & -m_{n+1}^{1/a,c}(f) & \cdots & (-1)^{k-2}m_{n+k-2}^{1/a,c}(f)& (-1)^{k}m_{n+k}^{1/a,c}(f)}{f\in F_2}
  \end{array}
  \right|.
\end{align}
Using the duality (\ref{sdm2b}), we have
\begin{align}\label{duomph}
\Omega _\F(n)&=\left(\frac{a}{a-1}\right)^{u_\F+n+k}\frac{(1+c)_{n+k-1}}{(n+k)!\kappa \xi _n}\Phi_n^\F, \\\label{dulaps}
\Lambda _\F(n)&=\left(\frac{a}{a-1}\right)^{u_\F+n+k-1}\frac{(1+c)_{n+k-2}}{(n+k-1)!\kappa \xi _n}\Psi_n^\F.
\end{align}
Taking into account (\ref{sdm3}) and according to Lemma \ref{sze}, as long as $\Phi_n^\F\not =0$, $n\ge 0$, the  polynomials  $q_n^\F$, $n\ge 0$, are orthogonal with respect to the measure
\begin{equation}\label{mraf}
\rho _{a,c}^{\F}=\sum _{x=u_\F}^\infty \prod_{f\in F_1}(x-f-u_\F)\prod_{f\in F_2}(x+c+f-u_\F)\frac{a^{x-u_\F}\Gamma(x+c-u_\F)}{(x-u_\F)!}\delta _x.
\end{equation}
Notice that the measure $\rho_{a,c}^{\F}$ is supported in the infinite set of nonnegative integers  $\sigma_\F$ (\ref{defsf}).

\begin{theorem}\label{th3.3} Let $\F=(F_1,F_2)$ be a pair of finite sets of positive integers. Then the polynomials $m_n^\F$ (\ref{defmex}), $n\in \sigma _\F$,
are common eigenfunctions of the second order difference operator
\begin{equation}\label{sodomex}
D_\F=h_{-1}(x)\Sh_{-1}+h_0(x)\Sh_0+h_1(x)\Sh_{1},
\end{equation}
where
\begin{align}\label{jpm1}
h_{-1}(x)&=\frac{x\Omega_\F(x+1)}{(a-1)\Omega_\F(x)},\\\label{jpm2}
h_0(x)&=-\frac{(1+a)(x+k)+ac}{a-1}+u_\F+\Delta\left(\frac{a(x+c+k-1)\Lambda_\F(x)}{(a-1)\Omega_\F(x)}\right),\\\label{jpm3}
h_1(x)&=\frac{a(x+c+k)\Omega_\F(x)}{(a-1)\Omega_\F(x+1)},
\end{align}
and $\Delta $ denotes the first order difference operator $\Delta f=f(x+1)-f(x)$. Moreover $D_\F(m_n^\F)=nm_n^\F$, $n\in \sigma_\F$.
\end{theorem}

\begin{proof}
The proof is similar to that of Theorem 3.3 in \cite{duch} but using here the three term recurrence relation for the Meixner polynomials
(\ref{Mxttrr}) and the dualities (\ref{duaqnrn}), (\ref{duomph}) and (\ref{dulaps}).
\end{proof}

\bigskip

The determinantal definition (\ref{defmex}) of the polynomials $m_n^\F$, $n\in \sigma_\F$, automatically implies a factorization for the corresponding difference operator $D_\F$ (\ref{sodomex}) in two difference operators of order $1$. This is a consequence of the Sylvester identity (see \cite{Gant}, pp. 32, or \cite{duch}, Lemma 2.1). This can be done by choosing one of the components of $\F=(F_1,F_2)$ and removing one element in the chosen component. An iteration of this procedure shows
that the polynomials $m_n^\F$, $n\in \sigma_\F$, and the corresponding difference operator $D_\F$  can be constructed by applying a sequence of at most $k$ Darboux transform (see the Definition \ref{dxt}) to the Meixner system (where $k$ is the number of elements of $\F$). We display the details in the following lemma, where we remove one element of the component $F_2$ of $\F$, and hence we have to assume $F_2\not =\emptyset$. A similar result can be proved by removing one element of the component $F_1$. The proof proceeds in the same way as the proof of Lemma 3.6 in \cite{duch} and it is omitted.

\begin{lemma}\label{lfe} Let $\F=(F_1,F_2)$ be a pair of finite sets of positive integers and assume $F_2\not =\emptyset$. We define the first order difference operators $A_\F$ and $B_\F$ as
\begin{align}
A_\F&=\frac{\Omega _\F(x+1)}{a\Omega_{\F_{2,\{ k_2\} }}(x+1)}\Sh_0-\frac{\Omega _\F(x)}{\Omega_{\F_{2,\{ k_2\} }}(x+1)}\Sh_1,
\\
B_\F&=\frac{ax\Omega _{\F_{2,\{ k_2\} }}(x+1)}{(a-1)\Omega_{\F}(x)}\Sh_ {-1}-\frac{a(x+c+k)\Omega _{\F_{2,\{ k_2\} }}(x)}{(a-1)\Omega_{\F}(x)}\Sh_0,
\end{align}
where $k_2$ is the number of elements of $F_2$ and the pair $\F_{2,\{ k_2\} }$ is defined by (\ref{deff2}).
Then $m_n^\F(x)=A_\F(m_{n-f_{k_2}^2+k_2-1}^{\F_{2,\{ k_2\} }})(x)$, $n\in \sigma_\F$. Moreover
\begin{align*}
D_{\F_{2,\{ k_2\} }}&=B_\F A_\F-(c+f_{k_2}^2-u_{\F_{2,\{ k_2\} }})Id,\\
D_{\F}&=A_\F B_\F-(c+f_{k_2}^2-u_{\F})Id.
\end{align*}
In other words, the system $(D_F,(m_n^\F)_{n\in \sigma _\F})$ can be obtained by applying a Darboux transform to the system
$(D_{\F_{2,\{ k_2\} }},(m_n^{\F_{2,\{ k_2\} }})_{n\in \sigma _{\F_{2,\{ k_2\} }}})$.
\end{lemma}

Analogous factorization can be obtained by removing instead of $f_{k_2}^2$ any other element $f_i^2$ of $F_2$, $1\le i<k_2$.

\section{Exceptional Meixner polynomials}
Given  real numbers $a,c$, with $a\not = 0,1$ and $c\not =0,-1,-2,\ldots$, in the previous Section we have associated to each pair $\F =(F_1, F_2)$ of finite sets of positive integers the polynomials $m_n^{a,c;\F}$, $n\in \sigma_\F$, which are always eigenfunctions of a second order difference operator with rational coefficients.
We are interested in the cases when, in addition, those polynomials are orthogonal and complete with respect to a positive measure.

\begin{definition} The polynomials $m_n^{a,c;\F}$, $n\in \sigma_\F$, defined by (\ref{defmex}) are called exceptional Meixner polynomials, if they are orthogonal and complete with respect to a positive measure.
\end{definition}

As we point out in Section \ref{sectadm}, the key concept for the construction of exceptional Meixner polynomials is that of
admissibility (see the Definition \ref{dadm}). The admissibility of $c$ and $\F$ can also be characterized in terms of the measure $\rho_{a,c}^\F$ (\ref{mraf}) and the sign of the Casorati polynomial $\Omega _\F ^{a,c}$ in $\NN$.

\begin{lemma}\label{l3.1} Given  real numbers $a,c$, with $0<a<1$ and $c\not =0,-1,-2,\ldots$, and a pair $\F $ of finite sets of positive integers, the following conditions are equivalent.
\begin{enumerate}
\item The measure $\rho_{a,c}^\F$ (\ref{mraf}) is positive.
\item $c$ and $\F$ are admissible.
\item $\Gamma(n+c+k)\Omega_\F ^{a,c}(n)\Omega_\F ^{a,c}(n+1)>0$ for all nonnegative integer $n$, where the polynomial $\Omega_\F^{a,c}$ is defined by (\ref{defom}).
\end{enumerate}
\end{lemma}

\begin{proof}
As in Section \ref{sectadm}, write $\hat c=\max \{-[c],0\}$. We then have
$$
\Gamma(x+c-u_\F)=\frac{\Gamma(x+c+\hat c-u_\F)}{(x+c)_{\hat c}}.
$$
Since $x+c+\hat c-u_\F> 0$, for $x\ge u_\F$, the equivalence between (1) and (2) is an easy consequence of the definitions of admissibility (\ref{defadm}) and of the measure $\rho _{a,c}^\F$.

We now prove the equivalence between (1) and (3).

(1) $\Rightarrow$ (3). Since the measure $\rho_{a,c}^\F$ is positive, the polynomials $(q_n^\F)_n$ (\ref{defqnme}) are orthogonal with respect to the measure $\rho_{a,c}^\F$ and have positive $L^2$-norm. According to (\ref{n2q}), we have
\begin{equation}\label{nssu}
\langle q_n^\F,q_n^\F\rangle =\frac{(-1)^k n!}{(n+k)!}\langle m_n^{a,c},m_n^{a,c}\rangle \Phi_n^\F\Phi_{n+1}^\F=
\frac{(-1)^k a^n\Gamma(n+c)}{(1-a)^{2n+c}(n+k)!}\Phi_n^\F\Phi_{n+1}^\F.
\end{equation}
We deduce then that $(-1)^k\Gamma(n+c)\Phi_n^\F\Phi_{n+1}^\F>0$ for all $n$. Using the duality (\ref{duomph}) and the definition of
$\xi_n$ in Lemma \ref{lem3.2}, we conclude that the sign of $(-1)^k\Gamma(n+c)\Phi_n^\F\Phi_{n+1}^\F$ is equal to the sign of $\Gamma(n+c+k)\Omega_\F^{a,c}(n)\Omega_\F^{a,c}(n+1)$. This proves (3).

(3) $\Rightarrow$ (1). Using Lemma \ref{sze}, the duality (\ref{duomph}), the definition of
$\xi_n$ in Lemma \ref{lem3.2} and proceeding as before, we conclude that the polynomials $(q_n^\F)_n$
are orthogonal with respect to $\rho_{a,c}^\F$ and have positive $L^2$-norm. This implies that there exists a positive measure $\mu$ with respect to which the polynomials $(q_n^\F)_n$ are orthogonal. Taking into account that the Fourier transform $H(z)$ of $\rho_{a,c}^\F$, defined by $H(z)=\int e^{-ixz}d\rho_{a,c}^\F(x)$, is an analytic function in the half plane $\Im z<-\log a$, and using moment problem standard techniques, it is not difficult to prove that $\mu$ has to be equal to $\rho _{a,c}^\F$. Hence the measure $\rho_{a,c}^\F$ is positive.
\end{proof}

According to the part 1 of Lemma \ref{ladm} and Lemma \ref{l3.1}, if $c$ and $\F$ are admissible, we have $c+k>0$ and $\Gamma (n+c+k)\Omega_\F^{a,c} (n)\Omega_\F ^{a,c}(n+1)>0$, for all $n\in \NN$. One can then deduce that if $c$ and $\F$ are admissible, then $\Omega_\F ^{a,c}(n)\Omega_\F^{a,c} (n+1)>0$, for all $n\in \NN$. We point out that the converse is  not true. Indeed, take $a=1/2$, $c=-7/2$, $F_1=\{1\}$, $F_2=\emptyset$.
A straightforward computation gives
$$
\Omega_\F^{a,c} (n)\Omega_\F^{a,c} (n+1)=\frac{(2n+7)(2n+9)}{4}>0,\quad n\in \NN .
$$
However, it is easy to see that $c$ and $\F$ are not admissible ((\ref{defadm}) is negative for $x=0,3$).

\bigskip

In the two following Theorems we prove that when  $c$ and $\F$ are admissible the polynomials $m_n^{a,c;\F}$, $n\in \sigma _\F$, are orthogonal and complete with respect to a positive measure.

\begin{theorem}\label{th4.4} Let $\F$ be a pair of finite sets of positive integers satisfying that $\Omega_\F^{a,c}(n)\not=0$ for all nonnegative integer $n$. Assume $-1<a<1$, $a\not =0$ and $c\not =0,-1,-2,\cdots $. Then the  polynomials $m_n^{a,c,\F}$, $n\in \sigma _\F$,
are orthogonal with respect to the (possibly signed) discrete measure
\begin{equation}\label{momex}
\omega_{a,c}^\F=\sum_{x=0}^\infty \frac{a^x\Gamma(x+c+k)}{x!\Omega_\F^{a,c}(x)\Omega_\F^{a,c}(x+1)}\delta_x.
\end{equation}
Moreover, for $-1<a<0$ the measure $\omega_{a,c}^{\F}$ is never positive, and for $0<a<1$ the measure $\omega_{a,c}^{\F}$ is positive if and only if $c$ and $\F$ are admissible.
\end{theorem}

\begin{proof}
Write $\Aa$ for the linear space generated by the polynomials $m_n^{a,c;\F}$, $n\in \sigma _\F$.
Using Lemma 2.5 of \cite{duch}, the definition of the measure $\omega_{a,c}^{\F}$ and the expressions for the difference coefficients of the operator $D_\F$ (see Theorem \ref{th3.3}), it is straightforward to check that $D_\F$  is symmetric with respect to the pair $(\omega_{a,c}^{\F},\Aa )$.
Since the polynomials $m_n^{a,c;\F}$, $n\in \sigma_\F$, are eigenfunctions
of $D_\F$ with different eigenvalues, Lemma 2.4 of \cite{duch} implies that they are orthogonal with respect to $\omega_{a,c}^{\F}$.

If $-1<a<0$ and the measure $\omega_{a,c}^{\F}$ is positive, since $\Gamma(n+c+k)$ is positive for $n$ big enough, we conclude that $\Omega_F^{a,c}(2n+1)\Omega_F^{a,c}(2n+2)<0$ for  $n$ big enough. But this would imply that $\Omega_\F^{a,c}$ has infinitely many real roots, which it is impossible since $\Omega_\F^{a,c}$ is a polynomial.

If $0<a<1$, according to Lemma \ref{l3.1}, $c$ and $\F$ are admissible if and only if $\Gamma(x+c+k)\Omega_F^{a,c}(x)\Omega_F^{a,c}(x+1)>0$ for all nonnegative integer $x$.

\end{proof}

\begin{theorem}\label{th4.5} Given  real numbers $a,c$, with $0<a<1$ and $c=0,-1,-2,\cdots $, and a pair $\F $ of finite sets of positive integers, assume that $c$ and $\F$ are admissible. Then the linear combinations of the  polynomials $m_n^{a,c;\F}$, $n\in \sigma _\F$, are dense in $L^2(\omega_{a,c}^{\F})$, where $\omega_{a,c}^{\F}$ is the positive measure (\ref{momex}). Hence $m_n^{a,c;\F}$, $n\in \sigma _\F$, are exceptional Meixner polynomials.
\end{theorem}

\begin{proof}
Using Lemma \ref{l3.1} and taking into account that $c$ and $\F$ are admissible, it follows that the measure $\rho _{a,c}^{\F}$ (\ref{mraf}) is positive. We remark that this positive measure is also determinate (that is, there is not other measure with the same moments as those of $\rho _{a,c}^{\F}$). As we pointed out above, this can be proved using moment problem standard techniques (taking into account, for instance, that the Fourier transform of $\rho_{a,c}^\F$ is an analytic function in the half plane $\Im z<-\log a$). Since for determinate measures the polynomials are dense in the associated $L^2$ space, we deduce that the sequence $(q_n^\F/\Vert q_n^\F\Vert _2)_n$ (where $q_n^\F$ is the polynomial defined by (\ref{defqnme})) is an orthonormal basis in $L^2(\rho_{a,c}^\F)$.

For $s\in \sigma _\F$, consider the function $h_s(x)=\begin{cases} 1/\rho _{a,c}^{\F}(s),& x=s\\ 0,& x\not =s, \end{cases}$ where by $\rho _{a,c}^{\F}(s)$ we denote the mass  of the discrete measure $\rho_{a,c}^\F$ at the point $s$.
Since the support of $\rho _{a,c}^{\F}$ is $\sigma_\F$, we get that $h_s\in L^2(\rho_{a,c}^\F)$. Its Fourier coefficients with respect to the orthonormal basis $(q_n^\F/\Vert q_n^\F\Vert _2)_n$ are $q_n^\F(s)/\Vert q_n^\F\Vert _2$, $n\ge 0$. Hence
\begin{equation}\label{pf1}
\sum _{n=0}^\infty \frac{q_n^\F(s)q_n^\F(r)}{\Vert q_n^\F\Vert _2 ^2}=\langle f_s,f_r\rangle _{\rho_{a,c}^\F}=\frac{1}{\rho_{a,c}^\F(s)}\delta_{s,r}.
\end{equation}
This is the dual orthogonality associated to the orthogonality
$$
\sum_{u\in \sigma _\F}q_n^\F(u)q_m^\F(u)\rho _{a,c}^{\F}(u)=\langle q_n^\F,q_n^\F\rangle \delta _{n,m}
$$
of the polynomials $q_n^\F$, $n\ge 0$, with respect to the positive measure $\rho _{a,c}^{\F}$ (see, for instance, \cite{At}, Appendix III, or \cite{KLS}, Th. 3.8).

Using (\ref{nssu}), (\ref{norme}) and the duality (\ref{duomph}), we get
\begin{equation}\label{pf0}
\frac{1}{\langle q_n^\F,q_n^\F\rangle _{\rho_{a,c}^\F}}=\omega_{a,c}^{\F}(n)x_n,
\end{equation}
where $x_n$ is the positive number given by
\begin{equation}\label{pf2}
x_n=\left(\frac{a}{a-1}\right)^{2u_\F+2k+1}\frac{(-1)^k(1-a)^cn!(1+c)_{n+k-1}(1+c)_{n+k}}{\kappa ^2\Gamma(n+c)\Gamma(n+c+k)(n+k+1)!\xi_n\xi_{n+1}},
\end{equation}
and $\kappa$ and $\xi_n$ are defined in Lemma \ref{lem3.2}.

Using now the duality (\ref{duaqnrn}), we can rewrite (\ref{pf1}) for $s=r$ as
\begin{equation}\label{pf3}
\sum _{n=0}^\infty \omega_{a,c}^{\F}(n)(m_r^\F(n))^2\kappa^2x_n\xi_n^2\zeta_r^2=\frac{1}{\rho_{a,c}^\F(r)}.
\end{equation}
A straightforward computation using (\ref{pf2}) and the definitions of $\kappa$, $\xi_n$ and $\zeta_r$ in Lemma \ref{lem3.2} gives
\begin{equation}\label{pf4}
\kappa^2x_n\xi_n^2\zeta_r^2=\frac{(1-a)^{c+2r-2u_\F-k}}{a^{k_1-2k}(\rho_{a,c}^\F(r))^2}.
\end{equation}
Inserting it in (\ref{pf3}), we get
\begin{equation}\label{pf5}
\langle m_r^\F,m_r^\F\rangle_{\omega_{a,c}^{\F}}=\frac{a^{k_1-2k}}{(1-a)^{c+2r-2u_\F-k}}\rho_{a,c}^\F(r).
\end{equation}
Consider now a function $f$ in $L^2(\omega_{a,c}^{\F})$ and write $g(n)=f(n)/x_n^{1/2}$, where $x_n$ is the positive number given by (\ref{pf2}). Using (\ref{pf0}), we get
$$
\sum_{n=0}^\infty\frac{\vert g(n)\vert ^2}{\langle q_n^\F,q_n^\F\rangle _{\rho_{a,c}^\F}}=\sum_{n=0}^\infty \omega_{a,c}^{\F}(n)\vert f(n)\vert ^2=\Vert f\Vert _2^2<\infty.
$$
Define now
$$
v_r=\sum_{n=0}^\infty\frac{g(n)q_n^\F (r)}{\langle q_n^\F,q_n^\F\rangle _{\rho_{a,c}^\F}}.
$$
Using Theorem III.2.1 of \cite{At}, we get
\begin{equation}\label{pf6}
\Vert f\Vert _2^2=\sum_{n=0}^\infty\frac{\vert g(n)\vert ^2}{\langle q_n^\F,q_n^\F\rangle _{\rho_{a,c}^\F}}=\sum _{r\in \sigma _\F}\vert v_r\vert ^2\rho_{a,c}^\F (r).
\end{equation}
On the other hand, using the dualities (\ref{duomph}) and (\ref{duaqnrn}), and (\ref{pf2}) and (\ref{pf3}), we have
$$
v_r=\frac{1}{(\rho_{a,c}^\F(r))^{1/2}}\sum_{n=0}^\infty f(n)\frac{m_r^\F(n)}{\Vert m_r^\F\Vert _2}\omega_{a,c}^{\F}(n).
$$
This is saying that $(\rho_{a,c}^\F(r))^{1/2}v_r$, $r\in \sigma _\F$, are the Fourier coefficients of $f$ with respect to the orthonormal system
$(m_n^\F/\Vert m_n^\F\Vert_2)_n$. Hence, the identity (\ref{pf6}) is Parseval's identity for the function $f$. From where we deduce that the
orthonormal system $(m_n^\F/\Vert m_n^\F\Vert_2)_n$ is complete in $L^2(\omega_{a,c}^{\F})$.

\end{proof}

\section{Constructing polynomials which are eigenfunctions of second order differential operators}
One can construct exceptional Laguerre polynomials by taking limit in the exceptional Meixner polynomials. We use the basic limit
(\ref{blmel}).

Given a pair $\F=(F_1,F_2)$ of finite sets of positive integers, using the expression (\ref{defmexa}) for the polynomials $m_n^{a,c;\F}$, $n\in\sigma_\F$, setting $x\to x/(1-a)$ and $c=\alpha +1$ and taking limit as $a\to 1$, we get (up to normalization constants) the polynomials, $n\in \sigma _\F$,
\begin{equation}\label{deflax}
L_n^{\alpha ;\F}(x)= \left|
  \begin{array}{@{}c@{}cccc@{}c@{}}
    & L_{n-u_\F}^{\alpha}(x)&(L_{n-u_\F}^{\alpha})'(x)&\cdots &(L_{n-u_\F}^{\alpha})^{(k)}(x) & \\
    \dosfilas{ L_{f}^{\alpha}(x) & (L_{f}^{\alpha})'(x) &\cdots  & (L_{f}^{\alpha})^{(k)}(x) }{f\in F_1} \\
    \dosfilas{ L_{f}^{\alpha}(-x) & L_{f}^{\alpha+1}(-x) & \cdots & L_{f}^{\alpha +k}(-x) }{f\in F_2}
  \end{array}
  \right|.
\end{equation}

More precisely
\begin{equation}\label{lim1}
\lim_{a\to 1}(a-1)^{n-(k_1+1)k_2}m_n^{a,c;\F}\left(\frac{x}{1-a}\right)=(-1)^{\binom{k+1}{2}+\sum_{f\in F_2}f}L_n^{\alpha ;\F}(x)
\end{equation}
uniformly in compact sets.

Notice that $L_n^{\alpha ;\F}$ is a polynomial of degree $n$ with leading coefficient equal to
$$
(-1)^{n-u_\F+\sum_{f\in F_1}f} \frac{V_{F_1}V_{F_2}\prod_{f\in F_1}(f-n+u_F)}{(n-u_\F)!\prod_{f\in F_1}f!\prod_{f\in F_2}f!},
$$
where $V_F$ is the Vandermonde determinant defined by (\ref{defvdm}).

We introduce the associated polynomials
\begin{equation}\label{defhom}
\Omega _{\F}^{\alpha}(x)=\left|
  \begin{array}{@{}c@{}cccc@{}c@{}}
    \dosfilas{ L_{f}^{\alpha}(x) & (L_{f}^{\alpha})'(x) &\cdots  & (L_{f}^{\alpha})^{(k-1)}(x) }{f\in F_1} \\
    \dosfilas{ L_{f}^{\alpha}(-x) & L_{f}^{\alpha+1}(-x) & \cdots & L_{f}^{\alpha +k-1}(-x) }{f\in F_2}
  \end{array}
  \right| .
\end{equation}
Notice that $\Omega_{\F}^{\alpha}$ is a polynomials of degree $u_\F+k_1$. To simplify the notation we sometimes write $\Omega_\F=\Omega_\F^{\alpha }$.

When $F_1=\emptyset$, using (\ref{Lagab}), we have for $\Omega _\F^\alpha$ the identity
\begin{align}\label{defhomf2v}
\Omega _{\F}^{\alpha}(x)&=\left|
  \begin{array}{@{}c@{}cccc@{}c@{}}
    \dosfilas{ L_{f}^{\alpha}(-x) & -(L_{f}^{\alpha})'(-x) & \cdots & (-1)^{k-1}(L_{f}^{\alpha })^{(k-1)}(-x) }{f\in F_2}
  \end{array}
  \right| .
\end{align}
We also straightforwardly have
\begin{equation}\label{rromh}
L_{u_F}^{\alpha ;\F}(x)=(-1)^{\binom{s_\F}{2}+s_\F k_1}\Omega_{\F_\Downarrow }^{\alpha +s_\F}(x),
\end{equation}
where the positive integer $s_\F$ and the pair $\F_\Downarrow$ are defined by (\ref{defs0f}) and (\ref{deffd}), respectively.

We will need to know the value at $0$ of the polynomial $\Omega_\F^{\alpha }$.

\begin{lemma}\label{v0le}
Let $\F$ be a pair of finite sets of positive integers, then $\Omega_\F^{\alpha }(0)$ is a polynomial in $\alpha$ of degree $u_\F+k_1$ which does not vanish in $\RR \setminus \{-1,-2,\ldots \}$. Moreover
\begin{align}\label{v0l}
\Omega_\F^{\alpha }(0)&=(-1)^{\binom{k_1}{2}}
\frac{\prod_{j=1}^2 V_{F_j}\prod _{i=1}^{k_j}(\alpha +i)_{k_j-i+1}\prod _{f\in F_j}(\alpha +k_j+1)_{f-k_j}}{\prod _{f\in F_1}f!\prod _{f\in F_2}f!\prod _{i=1}^{\min\{k_1,k_2\}}(\alpha +i)_{k_1+k_2-2i+1}}\\\nonumber &\hspace{4cm}\times
\prod_{f\in F_1}\prod _{g\in F_2}(\alpha +f+g+1).
\end{align}
\end{lemma}

\begin{proof}
The proof of (\ref{v0l}) follows by a carefully computation using that $L_n^\alpha (0)=\frac{(1+\alpha)_n}{n!}$ and standard determinant techniques.
Because of the value above of the Laguerre polynomials at $0$, $\Omega_\F^{\alpha }(0)$ is clearly a polynomial in $\alpha $; one can also see that the right hand side of (\ref{v0l}) is a polynomial because each factor of the form $\alpha +s$ in the denominator cancels with one in the numerator. It is now easy to see that the right hand side of (\ref{v0l}) only vanishes in some negative integers.

\end{proof}

Passing again to the limit, we can transform the second order difference operator (\ref{sodomex}) in a second order differential operator with respect to which the polynomials $L_n^{\alpha ;F}$, $n\in\sigma_\F$, are eigenfunctions.

\begin{theorem}\label{th5.1} Given a real number $\alpha \not =-1,-2,\cdots $ and a pair $\F$ of finite sets of positive integers, the polynomials $L_n^{\alpha;\F}$, $n\in \sigma _\F$,
are common eigenfunctions of the second order differential operator
\begin{equation}\label{sodolax}
D_F=x\partial ^2+h_1(x)\partial+h_0(x),
\end{equation}
where $\partial=d/dx$ and
\begin{align}\label{jph1}
h_1(x)&=\alpha +k+1-x-2x\frac{\Omega_\F'(x)}{\Omega_\F(x)},\\\label{jph2}
h_0(x)&=-k_1-u_\F +(x-\alpha -k)\frac{\Omega_\F'(x)}{\Omega_\F(x)}+x\frac{\Omega_\F''(x)}{\Omega_\F(x)}.
\end{align}
More precisely $D_\F(L_n^\F)=-nL_n^\F(x)$.
\end{theorem}

\begin{proof}
We omit the proof because proceeds as that of Theorem 5.1 in \cite{duch} and it is a matter of calculation using carefully the basic limits (\ref{blmel}) and its consequences
\begin{align}\label{lim2}
\lim _{a\to 1^-}(1-a)^{\beta_\F}\Omega _\F^{a,c}(x_a)&=(-1)^{\epsilon _\F}\Omega _\F^\alpha (x),\\\label{lim3}
\lim _{a\to 1^-}(1-a)^{\beta_\F-1}(\Omega _\F^{a,c}(x_a+1)-\Omega _\F^{a,c}(x_a))&=(-1)^{\epsilon _\F}(\Omega _\F^\alpha )' (x),\\\nonumber
\lim _{a\to 1^-}(1-a)^{\beta_\F-2}(\Omega _\F^{a,c}(x_a+1)-2\Omega _\F^{a,c}(x_a)+\Omega _\F^{a,c}(x_a-1))&=(-1)^{\epsilon _\F}(\Omega _F^\alpha )'' (x).
\end{align}
where $c=\alpha +1$, $\beta_\F=u_\F+k_1(1-k_2)$, $x_a=x/(1-a)$ and $\epsilon _\F=\sum_{f\in F_1}f$.

\end{proof}

To prove the completeness of the exceptional Laguerre polynomials in the associated $L^2$ space, we will need the following characterization of the linear space generated by $L_n^{\alpha;\F}$, $n\in \sigma _\F$.

\begin{lemma}\label{cpap}
Given a real number $\alpha \not =-1,-2,\ldots $ and a pair $\F$ of finite sets of positive integers, consider the linear space $\Aa$ generated by the polynomials $L_n^{\alpha;\F}$, $n\in \sigma _\F$. Then $p\in \Aa $ if and only if
\begin{equation}\label{lum}
(-2xp'+(x-\alpha -k)p)\Omega _\F '+xp\Omega _\F ''
\end{equation}
is divisible by $\Omega _\F $.
\end{lemma}

\begin{proof}
Write $\Bb =\{ p\in \PP : D_\F(p)\in \PP\}$. From the definition of the second order differential operator $D_\F $ (\ref{sodolax}), one easily sees that
$p\in \Bb $ is and only if the polynomial (\ref{lum}) is divisible by $\Omega_\F$.

Since each polynomial $L_n^{\alpha; \F}$, $n\in \sigma _\F$, is an eigenfunction for $D_\F$, we get that $L_n^{\alpha; \F}\in \Bb$ and hence $\Aa\subset \Bb$.

Consider the set of nonnegative integer $S=\NN \setminus \sigma_{\F}$. The definition of $\sigma_\F$ (\ref{defsf}) shows that
$S$ is finite and $s=\vert S\vert = u_\F+k_1=\deg(\Omega _\F)$.
We can  write $\PP =\Aa \oplus H_1$, where $H_1=\langle x^j: j\in S \rangle$. Hence $\dim H_1=\deg(\Omega _\F)$.

On the other hand, observe that the divisibility of (\ref{lum}) by $\Omega _\F$ imposes $\deg(\Omega _\F)$
linearly independent homogeneous conditions on the coefficients of $p$. We can then construct linearly independent polynomials $q_j$, $j=1,\ldots , \deg(\Omega _\F)$, such that  $\PP=\Bb\oplus H_2$, where $H_2=\langle q_j: j=1,\ldots , \deg(\Omega _\F)\rangle$. Hence $\dim H_2=\deg(\Omega _\F)$.
Since $\Aa\subset \Bb$ and $\dim H_1=\dim H_2$, we get that actually $\Aa=\Bb$.

\end{proof}

Again passing to the limit, from the factorization in Lemma \ref{lfe} we can factorize the second order differential operator $D_\F$ as product of two first order differential operators.
This can be done by choosing one of the components of $\F=(F_1,F_2)$ and removing one element in the chosen component. An iteration  shows that the system $(D_\F, (L_n^{\alpha; \F})_{n\in \sigma_\F})$ can be constructed by applying a sequence of $k$ Darboux transforms to the Laguerre system (see the Definition \ref{dxt})). We display the details in the following lemma, where we remove one element of the component $F_2$, and hence we have to assume $F_2\not =\emptyset$. A similar result can be proved by removing one element of the component $F_1$.

\begin{lemma}\label{lfel} Let $\F=(F_1,F_2)$ be a pair of finite sets of positive integers and assume $F_2\not =\emptyset$.  We define the first order differential operators $A_\F$ and $B_\F$ as
\begin{align}
A_\F&=-\frac{\Omega _\F(x)}{\Omega_{\F_{2,\{ k_2\}}}(x)}\partial+\frac{\Omega _\F'(x)+\Omega _\F(x)}{\Omega_{\F_{2, \{ k_2\}}}(x)},\\
B_\F&=\frac{-x\Omega _{\F_{2,\{ k_2\}}}(x)}{\Omega_{\F}(x)}\partial+\frac{x\Omega '_{\F_{2,\{ k_2\}}}(x)-(\alpha+k)\Omega_{\F_{2,\{ k_2\}}}(x)}{\Omega_{\F}(x)},
\end{align}
where $k_2$ is the number of elements of $F_2$ and the pair $\F_{2,\{ k_2\} }$ is defined by (\ref{deff2}).
Then $L_n^{\alpha ;\F}(x)=A_\F(L_{n-f_{k_2}^2+k_2-1}^{\alpha; \F_{2,\{ k_2\}}})(x)$, $n\in \sigma_F$. Moreover
\begin{align*}
D_{\F_{2,\{ k_2\}}}&=B_\F A_\F+(\alpha+f_{k_2}^2-u_{\F_{2,\{ k_2\}}}+1)Id,\\
D_{\F}&=A_\F B_\F+(\alpha +f_{k_2}^2-u_\F+1)Id.
\end{align*}
\end{lemma}
\begin{proof}
The Lemma can be proved applying limits in Lemma \ref{lfe}.

\end{proof}

\section{Exceptional Laguerre polynomials}
In the previous Section, given a real number $\alpha\not=-1,-2,\cdots $, we have associated to each pair $\F$ of finite sets of positive integers the polynomials $L_n^{\alpha ;\F}$, $n\in \sigma_\F$,
which are always eigenfunctions of a second order differential operator with rational coefficients.
We are interested in the cases when, in addition, those polynomials are orthogonal and complete with respect to a positive measure.

\begin{definition} The polynomials $L_n^{\alpha ;F}$, $n\in \sigma_\F$, defined by (\ref{deflax}) are called exceptional Laguerre polynomials, if they are orthogonal and complete with respect to a positive measure.
\end{definition}

The following Lemma and Theorem show that the admissibility of $\alpha +1$ and $\F$ will be the key to construct exceptional Laguerre polynomials.

\begin{lemma}\label{lem6.2} Given a real number $\alpha \not =-1,-2,\cdots $ and a pair $\F$ of finite sets of positive integers, if $\alpha +1$ and $\F$ are admissible then $\alpha +k>-1$ and $\Omega_\F ^\alpha $ (\ref{defhom}) does not vanish in $[0,+\infty )$.
\end{lemma}
\begin{proof}

First of all, the part 1 of Lemma \ref{ladm} gives that $\alpha +k>-1$.

Lemma \ref{v0le} shows that for $\alpha \not =-1, -2, \cdots $, $\Omega_\F^\alpha (0)\not =0$. Hence it is enough to prove
that $\Omega_\F ^\alpha \not=0$ for $x>0$.

Write $c=\alpha +1$. For $0<a<1$, consider the measure $\tau _a^c$ defined by
$$
\tau _a^c=(1-a)^{c-k}\sum _{x=0}^\infty\frac{a^x\Gamma(x+c+k)(m_{u_\F}^{a,c;\F}(x))^2}{x!\Omega _\F^{a,c}(x)\Omega_\F^{a,c}(x+1)}\delta {y_{a,x}},
$$
where
\begin{equation}\label{defya}
y_{a,x}=(1-a)x.
\end{equation}
Since $\alpha+1=c$ and $\F$ are admissible, we have that the measure $\tau_a^c$ is positive (Theorem \ref{th4.5}).

Consider the positive integer $s_\F$ and the pair $\F_\Downarrow$ defined by (\ref{defs0f}) and (\ref{deffd}), respectively.
We need the following limits
\begin{align}\label{lm1}
\lim _{a\to 1^-}(1-a)^{u_\F+k_1(1-k_2)}\Omega _\F^{a,c}(x/(1-a))&=\epsilon _0\Omega _\F^\alpha (x),\\\label{lm11}
\lim _{a\to 1^-}(1-a)^{u_\F+k-(k_1+1)k_2}\Omega _\F^{a,c}(x/(1-a)+1)&=\epsilon _0\Omega _\F^\alpha (x),\\\label{lm2}
\lim _{a\to 1^-}(1-a)^{u_\F-(k_1+1)k_2}r_{u_F}^{a,c;\F}(x/(1-a))&=\epsilon _1\Omega _{\F_{\Downarrow}}^{\alpha +s_\F} (x),\\\label{lm3}
\lim _{a\to 1^-}\frac{(1-a)^{c+k-1}a^{x/(1-a)}\Gamma(x/(1-a)+c+k)}{\Gamma (x/(1-a)+1)}&=x^{\alpha+k}e^{-x},
\end{align}
uniformly in compact sets of $(0,+\infty )$, where the $\epsilon$'s are the signs
$$
\epsilon_0=(-1)^{\sum_{f\in F_1}f},\quad \epsilon_1=(-1)^{\sum_{f\in F_1}f+\binom{s_\F}{2}+s_\F k_1}.
$$
The  first limit is (\ref{lim2}). The second one is a consequence of
(\ref{lim3}). The third one is a consequence of (\ref{lim1}) and (\ref{rromh}). The forth one is consequence of the asymptotic behavior of $\Gamma(z+u)/\Gamma(z+v)$ when $z\to\infty$ (see \cite{EMOT}, vol. I (4), p. 47).

We proceed in three steps.

\bigskip

\noindent
\textit{First step.} Assume that $\alpha +1$ and $\F$ are admissible and that there exists $x_0>0$ with $\Omega _\F^\alpha (x_0) =0$. Then $\Omega _{\F _{\Downarrow}}^{\alpha +s_\F} (x_0)=0$.

We take a real number $v$ with $x_0<v$ such that $\Omega _\F^\alpha (x) \not =0$ for $x\in (x_0,v]$. For a real number $u$ with $x_0<u<v$, write $I=[u,v]$. Then $\Omega _\F^\alpha $ does not vanish in $I$. Applying Hurwitz's Theorem to the limits (\ref{lm1}) and (\ref{lm11}) we can choice a contable set $X=\{a_n: n\in \NN \}$ of numbers in $(0,1)$ with $\lim_n a_n=1$ such that $\Omega _\F^{a,c}(x/(1-a))\Omega _\F^{a,c}(x/(1-a)+1)\not =0$, $x\in I$ and $a\in X$.

Hence, we can combine the limits (\ref{lm1}), (\ref{lm11}), (\ref{lm2}) and (\ref{lm3}) to get
\begin{equation}\label{lm4}
\lim _{a\to 1;a\in X}h_a(x)=\epsilon_3h(x),\quad \mbox{uniformly in $I$},
\end{equation}
where
\begin{align*}
h_a(x)&=(1-a)^{c-k-1}\frac{a^{x/(1-a)}\Gamma(x/(1-a)+c+k)(m_{u_F}^{a,c;\F}(x/(1-a)))^2}{\Gamma (x/(1-a)+1)\Omega_\F^{a,c} (x/(1-a))\Omega_\F^{a,c} (x/(1-a)+1)},\\
h(x)&=\frac{x^{\alpha +k}e^{-x}(\Omega_{\F_{\Downarrow }}^{\alpha +s_\F})^2(x)}{(\Omega_\F^\alpha )^2(x)},
\end{align*}
and $\epsilon_3$ is again a sign.

We now prove that
\begin{equation}\label{lm5}
\lim _{a\to 1 ;a\in X}\tau _a^c(I)=\epsilon _3\int_{I}h(x)dx.
\end{equation}
To do that, write $I_a=\{ x\in \NN: u/(1-a)\le x\le v/(1-a)\}$.
The numbers $y_{a,x}$, $x\in I_a$, form a partition of the interval $I$ with $y_{a,x+1}-y_{a,x}=(1-a)$ (see (\ref{defya})). Since the function $h$ is continuous  in the interval $I$, we get that
$$
\int_{I}h(x)dx=\lim_{a\to 1; a\in X}S_a,
$$
where $S_a$ is the Cauchy sum
$$
S_a=\sum_{x\in I_a}h(y_{a,x})(y_{a,x+1}-y_{a,x}).
$$
On the other hand, since $x\in I_a$ if and only if $u\le y_{a,x}\le v$ (\ref{defya}), we get
\begin{align*}
\tau _a^c(I)&=(1-a)^{c-k}\sum _{x\in I_a}\frac{a^x\Gamma(x+c+k)(m_{u_\F}^{a,c;\F}(x))^2}{x!\Omega _\F^{a,c}(x)\Omega_\F^{a,c}(x+1)}=(1-a)\sum _{x\in I_a}h_a(y_{a,x})\\
&=\sum _{x\in I_a}h_a(y_{a,x})(y_{a,x+1}-y_{a,x}).
\end{align*}
The limit (\ref{lm5}) now follows from the uniform limit (\ref{lm4}).

The identity (\ref{pf5}) says that $\tau _a^c(\RR)=a^{k_1-2k}\gamma_\F^c$, where the positive constant
$$
\gamma_\F^c=\rho_{a,c}^\F(u_\F)=\Gamma(c)\prod_{f\in F_1}(-f)\prod_{f\in F_2}(c+f)
$$
does not depend on $a$.
This gives $\tau _a^c(I)\le a^{k_1-2k}\gamma_\F^c$. And so from the limit (\ref{lm5}) we get
$$
\int_{I}h(x)dx \le \frac{\gamma_\F^c}{\epsilon_3}.
$$
That is
$$
\int _u^v\frac{x^{\alpha +k}e^{-x}(\Omega_{\F_{\Downarrow }}^{\alpha +s_\F})^2(x)}{(\Omega_\F^\alpha )^2(x)}dx\le \frac{\gamma_\F^c}{\epsilon _3}.
$$
On the other hand, if $\Omega _{\F _{\Downarrow}}^{\alpha +s_\F} (x_0)\not =0$, since $\Omega_\F^\alpha (x_0)=0$ we get
$$
\lim _{u\to x_0^+}\int _u^v\frac{x^{\alpha +k}e^{-x}(\Omega_{F_{\Downarrow }}^{\alpha +s_\F})^2(x)}{(\Omega_F^\alpha )^2(x)}dx=+\infty.
$$
Hence $\Omega _{\F _{\Downarrow}}^{\alpha +s_\F} (x_0)=0$.

\bigskip

The proof of the Theorem proceeds now by induction on $\max F_1$.

\noindent
\textit{Second step.} Assume $\alpha +1$ and $\F$ are admissible and $\max F_1=-1$ (that is $F_1=\emptyset$). Then
\begin{equation}\label{pas2}
\Omega _\F^\alpha (x) \not =0,\quad x> 0.
\end{equation}

Since $F_1=\emptyset$, the part 3 of Lemma \ref{ladm} implies that the assumption $\alpha +1$ and $\F$ are admissible is equivalent to the assumption $\alpha >-1$.

We prove this step by induction on $k_2$. For $k_2=1$, we have that $F_2$ is a singleton $F_2=\{f\}$, and then $\Omega _\F^\alpha (x)=L^\alpha_f(-x)$. The usual properties of the zeros of Laguerre polynomials ($\alpha >-1$) imply (\ref{pas2}).

Assume now that (\ref{pas2}) holds for $k_2\le s$ and $\alpha >-1$, and take a finite set of positive integers $F_2$, with $k_2=s+1$ elements. We notice that according to the definition of $s_{F_1}$ (\ref{defs0}) for $F_1=\emptyset,$ we have $s_{F_1}=1$. Hence we also have $s_\F=1$ (see (\ref{defs0f})). If
there exists $x_0>0$ such that $\Omega _\F^\alpha (x_0)=0$, using the first step, we get that also $\Omega _{\F _{\Downarrow}}^{\alpha +1} (x_0)=0$. Since $F_1=\emptyset$, we have $\F=\F _{\Downarrow}$ (see (\ref{deffd})), and hence, we can conclude that $\Omega _\F^{\alpha +j} (x_0)=0$, $j=0,1,2,\ldots $

For a positive integer $m\ge \max F_2+1\ge s+1$ consider the $(s+1)\times m$ matrix
$$
M=\left(
  \begin{array}{@{}c@{}cccc@{}c@{}}
    \dosfilas{ L_{f}^{\alpha}(-x_0) & -(L_{f}^{\alpha})'(-x_0) & \cdots & (-1)^m(L_{f}^{\alpha })^{(m)}(-x_0) }{f\in F_2}
  \end{array}
  \right) .
$$
Write $c_i$, $i=1,\ldots, m$, for the columns of $M$ (from left to right). For $j\ge 0$, consider the $(s+1)\times s$ submatrix $M_j$ of $M$ formed by the consecutive columns $c_{j+i}$, $i=1,\cdots s$, of $M$. Using (\ref{defhomf2v}), we see that the minor of $M_j$ formed by its first $s$ rows is equal to $\Omega _{\F_{2,\{ s+1\} }}^{\alpha +j}(x_0)$ where the pair $\F_{2,\{ s+1\} }$ is defined by (\ref{deff2}). Since the set $F_2\setminus {\{ f^2_{s+1}\}}$ has $s$ elements and $\alpha +j>-1$, the induction hypothesis says that $\Omega _{\F_{2,\{ s+1\} }}^{\alpha +j}(x_0)\not =0$, and hence
the columns $c_{j+i}$, $i=1,\cdots s$, of $M$ are linearly independent. On the other hand, the consecutive columns $c_{j+i}$, $i=1,\cdots s+1$, of $M$ are linearly dependent because its determinant is equal to  $\Omega _{\F }^{\alpha +j}(x_0)=0$. Using Lemma \ref{rmc}, we conclude that $\rank M=s$.
Then there exist numbers $e_f$, $f\in F_2$, not all zero such that the  polinomial $p(x)=\sum_{f\in F_2}e_fL_{f}^{\alpha}(x)$ is non null and has a zero of multiplicity $m$ in $-x_0$. But  the polynomial $p$ has degree at most $\max F_2$, and since $m\ge \max F_2+1>\deg p$, this shows that $p=0$, which it is a contradiction. This proves the second step.

\bigskip

Assume now that $\alpha +1$ and $\F$ are admissible and
\begin{equation}\label{pas3}
\Omega _\F^\alpha (x) \not =0,\quad x> 0,
\end{equation}
holds for $\max F_1\le s$.

\bigskip

\noindent
\textit{Third step.} If $\max F_1=s+1$, then (\ref{pas3}) also holds.

Consider the pair $\F _{\Downarrow}=\{(F_1)_{\Downarrow},F_2\}$ defined by (\ref{deffd}).
Since $F_1\not =\emptyset$, we have that $\max (F_1)_{\Downarrow}\le s $. The part 4 of Lemma \ref{ladm} says that if $\alpha +1$ and $\F$ are admissible then $\alpha +1+s_\F$ and $\F _{\Downarrow}$ are admissible as well. The induction hypothesis (\ref{pas3}) then says that
$\Omega _{\F _{\Downarrow}}^{\alpha +s_\F} (x) \not =0$ for $x> 0$. The first step then gives that also $\Omega _\F^\alpha (x) \not =0$,
for $x> 0$.
\end{proof}

\bigskip
We guess that the converse of the previous Theorem is true. However, the condition
$\Omega_\F ^\alpha (x)\not =0$, $x\ge 0$, is not enough to guarantee the admissibility of $\alpha +1$ and $\F$. Indeed, consider $F_1=\{ 1\}$, $F_2=\emptyset$ and $\F =(F_1,F_2)$. The definition (\ref{defadm}) straightforwardly gives that $\alpha +1$ and $\F$ are admissible if and only if $-2<\alpha<-1$. On the other hand, it is also easy to see that $\Omega _\F^\alpha (x)=\alpha+1-x$, Hence $\Omega_\F^\alpha \not =0$, $x\ge 0$, as far as
$\alpha <-1$. Hence for $\alpha <-2$, $\Omega_\F^\alpha \not =0$, $x\ge 0$, but $\alpha +1$ and $\F$ are not admissible.

\begin{theorem}\label{th6.3} Given a real number $\alpha \not =-1,-2,\cdots $ and a pair $\F$ of finite sets of positive integers, if $\alpha +1$ and $\F$ are admissible then the polynomials $L_n^{\alpha ;\F}$, $n\in \sigma _\F$,
 are orthogonal with respect to the positive weight
\begin{equation}\label{molax}
\omega_{\alpha;\F}(x)=\frac{x^{\alpha +k}e^{-x}}{(\Omega_\F^\alpha(x))^2},\quad x>0,
\end{equation}
and their linear combinations  are dense in  $L^2(\omega_{\alpha; \F})$. Hence $L_n^{\alpha ;\F}$, $n\in \sigma _\F$, are
exceptional Laguerre polynomials.
\end{theorem}

\begin{proof}
First of all, notice that the positive weight (\ref{molax}) has finite moments of any order because since $\alpha +1$ and $\F$ are admissible we have $\alpha+k> -1$ (part 1 of Lemma \ref{ladm}).

Write $\Aa $ for the linear space generated by the polynomials $L_n^{\alpha ;\F}$, $n\in \sigma _\F$.
Using Lemma 2.6 of \cite{duch}, it is easy to check that the second order differential operator $D_\F$ (\ref{sodolax}) is symmetric with respect to the pair $(\omega _{\alpha ;\F}, \Aa)$ (\ref{molax}). Since the polynomials $L_n^{\alpha ;F}$, $n\in \sigma _\F$, are eigenfunctions of $D_\F$ with different eigenvalues Lemma 2.4 of \cite{duch} implies that they are orthogonal with respect to $\omega _{\alpha ;F}$.

In order to prove the completeness of $L_n^{\alpha;\F}$, $n\in \sigma _\F$, we proceed in two steps.

\bigskip
\noindent
\textsl{Step 1.} For each $r>0$ and $\alpha >-1$, the linear space $\{(1+x)^rp: p\in \PP\}$ is dense in $L^2(x^\alpha e^{-x})$.

Since $1+x^r>0$, $x>0$, this is equivalent to the density of $\PP$ in $L^2((1+x^r)x^\alpha e^{-x})$. But this follows straightforwardly taking into account that $(1+x^r)x^\alpha e^{-x}dx$, $x>0$, is a determinate measure.

\bigskip
\noindent
\textsl{Step 2.} $\Aa$ is dense in $L^2(\omega _{\alpha;\F})$.

Take a function $f\in L^2(\omega _{\alpha;\F})$. Write $r=\deg(\Omega _\F^\alpha)$ and define the function
$g(x)=(1+x^r)f(x)/(\Omega _\F^\alpha(x))^2$. Since $\alpha +1$ and $\F$ are admissible, we get from the previous lemma that
$\alpha +k>-1$ and $\Omega _\F^\alpha(x)\not =0$, $x\ge 0$. Hence $g\in L^2(x^{\alpha+k}e^{-x})$. Given $\epsilon >0$ and using the first step, we get a polynomial $p$ such that
\begin{equation}\label{dmp}
\int \vert g(x)-(1+x^r)p(x)\vert ^2x^{\alpha+k}e^{-x}dx<\epsilon.
\end{equation}
Write $\gamma=\inf \{(1+x^r)/\Omega _\F^\alpha(x),x\ge 0\}$. We then get
\begin{align*}
\int \vert g(x)-(1+x^r)p(x)\vert ^2 x^{\alpha+k}e^{-x}dx&=\int \left\vert \frac{1+x^r}{\Omega _\F^\alpha(x)}\right\vert ^2\vert f(x)-(\Omega _\F^\alpha(x))^2p(x)\vert ^2\omega _{\alpha;\F}dx\\
&\ge \gamma ^2 \int  \vert f(x)-(\Omega _\F^\alpha(x))^2p(x)\vert \omega _{\alpha;\F}(x)dx .
\end{align*}
Using (\ref{dmp}), we can conclude that the linear space $\{(\Omega _\F^\alpha(x))^2p: p\in \PP\}$ is dense in $L^2(\omega _{\alpha;\F})$.

Lemma \ref{cpap} gives that $\{(\Omega _\F^\alpha(x))^2p: p\in \PP\}\subset \Aa $. This proves the second step and the Theorem.

\end{proof}

\bigskip
Proceeding as in the first step of  Theorem \ref{lem6.2}, one can
find the norm of the polynomials $L_n^{\alpha;\F}$, $n\in
\sigma_\F$, from the norm of the polynomials $m_n^{a,c;\F}$ (see
(\ref{pf5})).

\begin{corollary} Given a real number $\alpha \not =-1,-2,\cdots $ and a pair $\F$ of finite sets of positive integers, assume that $\alpha +1$ and $\F$ are admissible.
Then for $n\in \sigma _\F$, we have
$$
\int _0^\infty (L_n^{\alpha ;\F}(x))^2 \frac{x^{\alpha
+k}e^{-x}}{(\Omega_\F^\alpha(x))^2}dx=\frac{\p_\F(n-u_\F)\Gamma(n-u_\F+\alpha+1)}{(n-u_\F)!},
$$
where $\p _\F(x)$ is the polynomial defined by $\p_\F(x)=\prod_{f\in
F_1}(x -f)\prod_{f\in F_2}(x+\alpha+f+1)$.
\end{corollary}

\section{Appendix}
When the determinants $\Omega _\F ^{a,c} (n)\not =0$ (\ref{defom}), $n\ge 0$ (or equivalently, $\Phi_n^\F\not =0$ (\ref{defphme}), $n\ge 0$),
the following alternative construction of the polynomial $q_n^{a,c;\F}$ (\ref{defqnme}) has been given in \cite{DdI}.

For a pair $\F=(F_1,F_2)$ of finite sets of positive integers, consider the involuted sets $I(F_1)=G_1$ and $I(F_2)=G_2$,
where the involution $I$ is defined by (\ref{dinv}) and the number $v_\F$ defined by (\ref{defvf}). Write $m=m_1+m_2$, where $m_1$, $m_2$ are the number of elements of $G_1$ and $G_2$, respectively.

Assume that $\Omega _\F^{a,c} (n)\not =0$, $n\ge 0$, and write $\tilde c=c+M_{F_1}+M_{F_2}+2$, where $M_F$ denotes de maximum element of $F$.
Using  Theorem 1.1 of \cite{DdI}, we have
\begin{equation}\label{qusmei}
q_n^{a,c;\F}(x)=\alpha_n
\left|
  \begin{array}{@{}c@{}cccc@{}c@{}}
    &m_n^{a,\tilde c}(x-v_\F)&(\frac{a}{a-1})m_{n-1}^{a,\tilde c}(x-v_\F)&\cdots &(\frac{a}{a-1})^{m}m_{n-m}^{a,\tilde c}(x-v_\F) & \\
    \dosfilas{ m_{g}^{a,2-\tilde c}(-n-1) & am_{g}^{a,2-\tilde c}(-n) &\cdots  & a^m m_{g}^{a,2-\tilde c}(-n+m-1) }{g\in G_1} \\
    \dosfilas{ m_{g}^{1/a,2-\tilde c}(-n-1) & m_{g}^{1/a,2-\tilde c}(-n) & \cdots & m_{g}^{1/a,2-\tilde c}(-n+m-1)}{g\in G_2}
  \end{array}
  \right|,
\end{equation}
where the positive integer $v_\F$ is defined by (\ref{defvf}) and $\alpha_n$, $n\ge 0$, is certain normalization constant.

The dualities (\ref{sdm2b}) and (\ref{duaqnrn}) then provides an alternative definition of the polynomial $m_n^{a,c;\F}$, $n\ge v_\F$ (\ref{defmex}). Indeed, write $r_j^c$, $j\ge 0$, for the polynomial of degree $m$ defined by $r_j^c(x)=(c+x-m)_{m-j}(x-j+1)_{j}$, and, as before, $\tilde c=c+M_{F_1}+M_{F_2}+2$. After an easy calculation, we conclude that
\begin{equation}\label{qusmei2}
m_n^{a,c;\F}(x)=\beta_n
\left|
  \begin{array}{@{}c@{}cccc@{}c@{}}
    &r_0^{\tilde c}(x)m_{n-v_\F}^{a,\tilde c}(x)&r_1^{\tilde c}(x)m_{n-v_\F}^{a,\tilde c}(x-1)&\cdots &r_m^{\tilde c}(x)m_{n-v_\F}^{a,\tilde c}(x-m) & \\
    \dosfilas{ m_{g}^{a,2-\tilde c}(-x-1) & am_{g}^{a,2-\tilde c}(-x) &\cdots  & a^mm_{g}^{a,2-\tilde c}(-x+m-1) }{g\in G_1} \\
    \dosfilas{ m_{g}^{1/a,2-\tilde c}(-x-1) & m_{g}^{1/a,2-\tilde c}(-x) & \cdots & m_{g}^{1/a,2-\tilde c}(-x+m-1)}{g\in G_2}
  \end{array}
  \right|,
\end{equation}
where $\beta_n$, $n\ge 0$, is certain normalization constant.

When the sum of the cardinalities of the involuted sets $G_1=I(F_1)$ and $G_2=I(F_2)$ is less than the sum of the cardinalities of $F_1$ and $F_2$,
the expression (\ref{qusmei2}) will provide a more efficient way than (\ref{defmex}) for an explicit computation of the polynomials $m_n^{a,c;\F}$, $n\ge v_{\F}$.
For instance, take $F_1=\{1,\cdots, k\}$, $F_2=\{1,\cdots, k-2, k\}$. Since $I(F_1)=\{ k\}$, $I(F_2)=\{ 1, k \}$, the determinant in (\ref{qusmei2}) has order $4$
while the determinant in (\ref{defmex}) has order $2k$.

\bigskip
Assume now that $\alpha +1$ and $\F$ are admissible (\ref{defadm}). Write $c=\alpha +1$. According to Lemma \ref{l3.1}, this gives for all $0<a<1$ that $\Gamma(x+c+k)\Omega ^{a,c} _\F(x)\Omega ^{a,c} _\F(x+1)>0$ for $x\in \NN $, where $\Omega _\F^{a,c}$ is the polynomial (\ref{defom}) associated to the Meixner family. In particular $\Omega ^{a,c} _\F(x)\not =0$, for all nonnegative integer $x$.
Write $w^{\alpha ,j} _n=j!\binom{n+\alpha }{j}$ and $\tilde \alpha =\alpha +M_{F_1}+M_{F_2}+2$.
Hence, if instead of (\ref{defmexa}) we take limit in (\ref{qusmei2}), we get the following alternative expression for the polynomials $L_n^{\alpha ;F}$, $n\ge v_\F$,
\begin{equation}\label{deflaxa}
L_n^{\alpha ;\F}(x)=\gamma_n
\left|
  \begin{array}{@{}c@{}cccc@{}c@{}}
    &x^mL_{n-v_\F}^{\tilde \alpha}(x)&w_{n-v_\F}^{\tilde \alpha,1} x^{m-1}L_{n-v_\F}^{\tilde \alpha-1}(x)&\cdots &w_{n-v_\F}^{\tilde \alpha,m} L_{n-v_\F}^{\tilde \alpha -m}(x) & \\
    \dosfilas{ L_{g}^{-\tilde \alpha}(-x) & L_{g}^{-\tilde \alpha +1}(-x) &\cdots  & L_{g}^{-\tilde \alpha +m}(-x) }{g\in G_1} \\
    \dosfilas{ L_{g}^{-\tilde \alpha }(x) & (L_{g}^{-\tilde \alpha})'(x) & \cdots & (L_{g}^{-\tilde \alpha})^{(m)}(x)}{g\in G_2}
  \end{array}
  \right|,
\end{equation}
where $\gamma_n$ is certain normalization constant.

Both determinantal constructions (\ref{qusmei2}) and
(\ref{deflaxa})  for the polynomials $m_{n}^{a,c;\F}$ and
$L_n^{\alpha:\F}$, $n\in \sigma_\F$, respectively, imply a couple
of factorizations of the second order difference and differential
operators $D_\F$ (see (\ref{sodomex}) and (\ref{sodolax}),
respectively) in two first order difference and differential
operators, respectively. These factorizations are different to the
factorizations displayed in Lemmas \ref{lfe} and \ref{lfel},
respectively. We do not include the details here but both
factorizations can be worked out as Lemmas 3.7 and 5.3 of
\cite{duch}.

\bigskip

\noindent
\textit{Mathematics Subject Classification: 42C05, 33C45, 33E30}

\noindent
\textit{Key words and phrases}: Orthogonal polynomials. Exceptional orthogonal polynomial. Difference operators. Differential operators.
Meixner polynomials. Krawtchouk polynomials. Laguerre polynomials.

     \end{document}